\documentclass[12pt,leqno]{amsart}
\usepackage{amsmath, amssymb, amscd, amsfonts,    stmaryrd, turnstile, mathrsfs, eucal,color}

\definecolor{dullmagenta}{rgb}{0.4,0,0.4}

\definecolor{darkblue}{rgb}{0,0,0.4}

\usepackage[colorlinks=true, pdfstartview=FitV, linkcolor=red, citecolor=blue, urlcolor=darkblue]
{hyperref}

\newtheorem{theorem}{Theorem}[section]
\newtheorem{lemma}[theorem]{Lemma}
\newtheorem{proposition}[theorem]{Proposition}
\newtheorem{corollary}[theorem]{Corollary}

\theoremstyle{definition}
\newtheorem{definition}[theorem]{Definition}
\newtheorem{example}[theorem]{Example}
\newtheorem{remark}[theorem]{Remark}

\begin{document}

\title[Pairs of rings sharing their units]{Pairs of rings sharing their units}

\author[G. Picavet and M. Picavet]{Gabriel Picavet and Martine Picavet-L'Hermitte}
\address{Math\'ematiques \\
8 Rue du Forez, 63670 - Le Cendre\\
 France}
\email{picavet.gm(at) wanadoo.fr}

\begin{abstract} We are working in the category of commutative unital rings and denote by $\mathrm U(R)$ the group of units of a nonzero ring $R$. An extension of rings $R\subseteq S$, satisfying $\mathrm U(R)=R \cap\mathrm U(S)$ is usually called local. This paper is devoted to the study of ring extensions such that $\mathrm U(R)=\mathrm U(S)$, that we call strongly local. P. M. Cohn in a paper, entitled Rings with zero divisors, introduced some strongly local extensions. We generalized under the name Cohn's rings his definition and give a comprehensive study of these extensions. As a consequence, we give a constructive proof of his main result. Now Lequain and Doering studied strongly local extensions, where $S$ is semilocal, so that $S/\mathrm J(S)$, where $\mathrm J(S)$ is the Jacobson radical of $S$, is Von Neumann regular. These  rings are usually called $J$-regular. We establish many results on $J$-regular rings in order to get substantial results on strongly local extensions when $S$ is $J$-regular. The Picard group of a $J$-regular ring is trivial, allowing to evaluate the group $\mathrm U(S)/\mathrm U(R)$ when  $R$ is $J$-regular. We then are able to give a complete characterization of the Doering-Lequain context. A Section is devoted to examples. In particular, when $R$ is a field, the strongly local and weakly strongly inert properties  are equivalent.
\end{abstract} 



\maketitle

\section{Introduction and Notation} This paper deals with commutative unital rings and their morphisms. Any ring $R$ is supposed nonzero. We denote by $\mathrm U(R)$ the set of units of a ring $R$. We will call {\it strongly local} an extension of rings $R\subseteq S$ such that $\mathrm U(R) = \mathrm U(S)$, also termed as an SL extension. The reason why is that a ring extension $R\subseteq S$ is classically called local if $\mathrm U(R) = \mathrm U(S)\cap R$. Naturally these notions do not coincide and have a ring morphism version. If $R\to S$ is a ring morphism and $Q$ is a prime ideal of $S$ lying over $P$ in $R$, then $R_P\to S_Q$ is usually a called a local morphism of $R\to S$.

Our work  takes its origin in the reading of two papers.
One of them was written by P.M. Cohn \cite{ZCOHN}: for any ring $R$, he exhibits an SL  extension $R\subseteq R'$, such that any non unit of $R'$ is a zerodivisor. To prove his main result, Cohn introduces some special rings in a lemma. We have considered rings of the same vein. The idea is as follows: if $I$ is an ideal of a ring $R$, we define the ring $R/\! \! /I:= R[X]/XI[X]$ (where $X$ is an indeterminate over $R$). This notation may seem weird, but it explains that the ring $R/I$ is shifted. When $I$ is a semiprime ideal, we have an SL extension $R\subseteq R/\! \! /I$ with very nice properties. To have a better understanding, consider a field $R$ and $I=0$, we recover $R \subseteq R[X]$. We reprove Cohn's result and give a constructive proof, not using a transfinite induction.  But Cohn's ring is not necessarily the same as ours, by lack of unicity. All these considerations are developed in Section 6. In Section 7 we consider a ring morphism $R \to R\{X\}$ used by E. Houston in the context of Noetherian rings and their dimensions. We generalize his results and give a link with the rings $R/\! \! /M$, where $M$ is a maximal ideal of $R$. 

The other one was written by Doering and Lequain \cite{DL}. This paper deals with pair of semilocal rings sharing their group of units. A first observation is that for a semilocal ring $R$ with Jacosbon radical $J$, the ring $R/J$ is Von Neumann regular-absolutely flat, in which case the ring is called in the literature $\mathrm J$-regular. Note that units are closely linked to Jacobson radicals. In order to generalize Lequain-Doering's results in a substantial way, we were lead to study $\mathrm J$-regular rings and their behavior with respect to ring morphisms, a subject not treated in the literature. Actually for an extension $R\subseteq S$, there is an exact sequence of Abelian groups $1\to\mathrm U(R)\to\mathrm U(S)\to\mathcal I(R,S)\to\mathrm{Pic}(R)$, where $\mathcal I(R,S)$ is the group of $R$-submodules of $S$ that are invertible. Now if $R$ is $\mathrm J$-regular, its Picard group is zero, so that $\mathcal I(R,S)$ measures the defect of strongly localness of the extension. A first crucial result is that for an SL extension $R\subseteq S$, then $S$ is $\mathrm J$-regular if and only if $R$ is $\mathrm J$-regular and $R\subseteq S$ is integral seminormal. If these conditions hold, then $\mathrm J(R)=\mathrm J(S)$. This material is developed in Section 5. The paper culminates in Section 8 with a substantial result in the Doering-Lequain style: an extension $R\subseteq S$, where $S$ is a semilocal ring, is SL if and only if $R\subseteq S$ is a seminormal integral FIP extension, whose residual extensions are isomorphisms and such that $R/M\cong\mathbb Z/2 \mathbb Z$ for each $M\in\mathrm{MSupp}(S/R)$. Now Section 2 examines the behavior of local and SL extensions and many examples are provided. In Section 3 we give a list of extensions that are local. Section 4 explores the properties of SL extensions. We give in Section 9 a series of examples of SL extensions, for example strongly inert extensions. These examples show that it seems impossible to find a general criterion for the SL property, except in the semilocal case.

As an example, we build at the end of the paper a strongly local ring morphism $f:\mathbb F_2[X]/(X^3-1)\to \mathbb F_2[X]/(X^4-X)$.

If $R\subseteq S$ is a (ring) extension, we denote by $[R,S]$ the set of all $R$-subalgebras of $S$. 
  
We will mainly consider ring morphisms that are ring extensions. A property $(P)$ of ring morphisms $f:R\to S$ is called universal if for any base change $R\to R'$, the ring morphism $R'\to R'\otimes_R S$ verifies $(P)$. 
As usual, Spec$(R)$ and Max$(R)$ are the set of prime and maximal ideals of a ring $R$. We denote by $\kappa_R(P)$ the residual field $R_P/PR_P$ at a prime ideal $P$ of $ R$. If $R\subseteq S$ is a ring extension and $Q\in\mathrm{Spec}(S)$, there exists a residual field extension $\kappa_R(Q\cap R)\to\kappa_S(Q)$. 

A {\it (semi-)local} ring is  a ring with (finitely many maximal ideals) one maximal ideal. 
  For an extension $R\subseteq S$ and an ideal $I$ of $R$, we write $\mathrm{V}_S(I):=\{P\in\mathrm{Spec}(S)\mid I\subseteq P\}$.
  The support of an $R$-module $E$ is $\mathrm{Supp}_R(E):=\{P\in\mathrm{Spec}(R)\mid E_P\neq 0\}$, and $\mathrm{MSupp}_R(E):=\mathrm{Supp}_R(E)\cap\mathrm{Max}(R)$. When $R\subseteq S$ is an extension, we will set $\mathrm{Supp}(T/R):=\mathrm {Supp}_R(T/R)$ and $\mathrm{Supp}(S/T):=\mathrm{Supp}_R(S/T)$ for each $T\in [R,S]$, unless otherwise specified. 
  
For a ring $R$, we denote by $\mathrm Z(R)$ the set of all zerodivisors of $R$, by $\mathrm{Nil}(R)$ the set of nilpotent elements of $R$
and by $\mathrm J(R)$ its Jacobson  radical. The Picard group of a ring $R$ is denoted by $\mathrm{Pic}(R)$.
   
Now  $(R:S)$ is the conductor of $R\subseteq S$. The integral closure of $R$ in $S$ is denoted by $\overline R^S$ (or by $\overline R$ if no confusion can occur).

A ring extension $R\subseteq S$ is called an {\it i-extension} if the natural map $\mathrm{Spec}(S)\to\mathrm{Spec}(R)$ is injective. 

An extension $R\subseteq S$ is said to have FIP (for the ``finitely many intermediate algebras property") or is an FIP  extension if $[R,S]$ is finite. 
We also say that the extension $R\subseteq S$ has FCP (or is an FCP extension) if each chain in $[R,S]$ is finite, or equivalently, its lattice is Artinian and Noetherian. An FCP extension is finitely generated, and (module) finite if integral. 

In the case of an FCP extension $R\subseteq S$ where $S$ is semilocal, we will show in Theorem \ref{5.05} that there is a greatest $T\in[R,S]$ such that $R\subseteq T$ is SL, that is $R\subseteq T$ is a unique MSL-subextension.

Finally, $|X|$ is the cardinality of a set $X$, $\subset$ denotes proper inclusion and for a positive integer $n$, we set $\mathbb{N}_n:=\{1,\ldots,n\}$. If $R$ and $S$ are two isomorphic rings, we will write $R\cong S$. If $M$ and $N$ are $R$-modules, we write $M\cong_R N$ if $M$ and $N$ are isomorphic as $R$-modules. 

\section{First properties of (strongly) local extensions} 

Generalizing the definition of a local ring morphism between local rings, a ring morphism $f= R\to S$ is called {\it local} if $f^{-1}(\mathrm U(S)) =\mathrm U(R)$.  We recover the case of a local extension of local rings $f: (R,M) \to (S,N)$ where $f^{-1}(N)=M$. We will mainly be concerned by ring extensions $R \subseteq S$, and in this case the extension is local (reflects units) if  $ \mathrm{U}(R)= \mathrm{U}(S)\cap R$.

We call a ring extension $R\subseteq S$  {\it strongly local (SL)} if $\mathrm{U}(R)= \mathrm{U}(S)$. A strongly local extension is obviously local and a  ring morphism $f:R\to S$ is called {\it strongly local (SL)} if $f(\mathrm U(R))=\mathrm U(S)$. 

Note that if $f$ is a surjective local morphism then $f$ is SL. Indeed we always have $f (\mathrm U(R))\subseteq\mathrm U(S)$. Assume that $f$ is surjective and local, and let $y\in\mathrm U(S)$. There exists some $a\in R$ such that $y=f(a)\in\mathrm U(S)$. It follows that $a\in f^{-1}(\mathrm U(S))=\mathrm U(R)$, so that $y\in f(\mathrm U(R))$, giving $\mathrm U(S)\subseteq f(\mathrm U(R))$, and then $f(\mathrm U(R))=\mathrm U(S)$, that is $f$  is SL. 

A first example of local extension is given by an $R$-module $M$ and its Nagata extension $f: R \to R\oplus M$, where $f(x) = (x,0)$. A unit $(a,m)$ of $R\oplus M$ 
 is such that $a$ is a unit of $R$ and $(a,m)$ has an inverse  of the form $(a',n)$, where $a'=a^{-1}$
 and $n= -a'^2m$. But this extension is not  SL. 

The extension $\mathbb Z \subseteq \mathbb Z[2i] $ is SL.

If $R$ is a ring then the extension $R[X^2] \subseteq R[X]$ is local but not SL 
 in general. For example, if $a\in R$ is nilpotent, then $1+aX\in \mathrm{U}(R[X])$ but $1+aX\not\in \mathrm{U}(R[X^2])$.

 \begin{proposition}\label{3.0} An extension $R\subseteq S$ such that $R$ and $S$ have the same prime ideals is local and is trivial if it is SL. 
\end{proposition} 
\begin{proof} We know that $R$ and $S$ are local rings with the same maximal ideal   \cite[Proposition 3.3]{AD}. Let $M$ be this common maximal ideal. As usual, we have $\mathrm{U}(R)\subseteq\mathrm{U}(S)\cap R$. Let $x\in\mathrm{U}(S)\cap R$. Then, $x\not\in P$, for any $P\in\mathrm{Spec}(S)=\mathrm{Spec}(R)$. Since $x\in R$, it follows that $x\in\mathrm{U}(R)$, so that $\mathrm{U}(R)=\mathrm{U}(S)\cap R$ and $R\subseteq S$ is local. If $R\subseteq S$ is SL, that is $\mathrm{U}(R)=\mathrm{U}(S)$, we have $R=M\cup\mathrm{U}(R)=M\cup\mathrm{U}(S)=S$.
\end{proof}

See \cite{AD} for examples and also PVD and $D+M$ construction.

 \begin{proposition}\label{3.1} Let $f:R\to S$ and $g:S\to T$ be two ring morphisms.  
\begin{enumerate}
\item If $f$ and $g$ are (strongly) local, so is $g\circ f$.
\item If $g\circ f$ is local, so is $f$. In case $f$ is surjective, then $g$ is local.
\item If $g\circ f$ is SL, so is $g$. In case $g$ is injective, then $f$ is SL.
\end{enumerate}
\end{proposition}
\begin{proof} (1) Assume that $f$ and $g$ are local. Then, $f^{-1}(\mathrm U(S))=\mathrm U(R)$ and $g^{-1}(\mathrm U(T))=\mathrm U(S)$, so that $(g\circ f)^{-1}(\mathrm U(T))=f^{-1}[g^{-1}(\mathrm U(T))]=f^{-1}(\mathrm U(S))=\mathrm U(R)$, so that $g\circ f$ is local.

Assume that $f$ and $g$ are SL. Then, $f(\mathrm U(R))=\mathrm U(S)$ and $g (\mathrm U(S))=\mathrm U(T)$, so that $(g\circ f)(\mathrm U(R))=g[f(\mathrm U(R))]=g(\mathrm U(S))=\mathrm U(T)$, giving that $g\circ f$ is  SL.

(2) Assume that $g\circ f$ is local. Then, $(g\circ f)^{-1}(\mathrm U(T))=\mathrm U(R)$, so that $f^{-1}[g^{-1}(\mathrm U(T))]=\mathrm U(R)$. Obviously, $\mathrm U(R)\subseteq f^{-1}(\mathrm U(S))$. Let $x\in f^{-1}(\mathrm U(S))$. It follows that $f(x)\in \mathrm U(S)$, whence $(g\circ f)(x)=g[f(x)]\in g(\mathrm U(S))\subseteq\mathrm U(T)$. To end, $x\in(g\circ f)^{-1}(\mathrm U(T))=\mathrm U(R)$ and $f^{-1}(\mathrm U(S)\subseteq \mathrm U(R)$, giving $f^{-1}(\mathrm U(S))=\mathrm U(R)$ and $f$ is local.

Assume that, moreover, $f$ is surjective. Obviously, $\mathrm U(S)\subseteq g^{-1}(\mathrm U(T))$. Let $y\in g^{-1}(\mathrm U(T))$, so that $g(y)\in\mathrm U(T)\ (*)$. But, $y\in S$ and $f$ surjective imply that there exists $x\in R$ such that $y=f(x)$. By $\ (*)$, we get $(g\circ f)(x)=g[f(x)]=g(y)\in\mathrm U(T)$, 
from which it follows that 
 $x\in(g\circ f)^{-1}(\mathrm U(T))=\mathrm U(R)$ and $y=f(x)\in f(\mathrm U(R))\subseteq\mathrm U(S)$. To end, $\mathrm U(S)= g^{-1}(\mathrm U(T))$ and $g$ is local.
 
(3) Assume that $g\circ f$ is SL. Then, $(g\circ f)(\mathrm U(R))=\mathrm U(T)$. Obviously, $g[\mathrm U(S)]\subseteq\mathrm U(T)$. Let $y\in\mathrm U(T)$. There exists $x\in\mathrm U(R)$ such that $y=(g\circ f)(x)=g[f(x)]\in g[f(\mathrm U(R))]\subseteq g[\mathrm U(S)]$, which gives $\mathrm U(T)\subseteq g[\mathrm U(S)]$ and $g[\mathrm U(S)]=\mathrm U(T)$. Then,  $g$ is SL.

Assume that, moreover, $g$ is injective. Obviously, $f[\mathrm U(R)]\subseteq\mathrm U (S)$. Let $z\in\mathrm U(S)$. Then, $g(z)\in\mathrm U(T)=(g\circ f)(\mathrm U(R))$, so that there exists $x\in\mathrm U(R)$ such that $g(z)=(g\circ f)(x)=g[f(x)]$. Since $g$ is injective, it follows that $z=f(x)\in f[\mathrm U(R)]$ and $\mathrm U(S)\subseteq f[\mathrm U(R)]$. To end, $\mathrm U(S)= f(\mathrm U(R))$ and $f$ is  SL.
\end{proof}

\begin{remark}\label{3.11} Let $f:R\to S$  be a ring morphism. Since the ring morphism $R/\ker(f)\to S$ associated to $f$ is injective and the canonical ring morphism $R\to R/\ker(f)$ is surjective, we may consider the extension $R/\ker(f)\subseteq S$. Then, Proposition \ref{3.1} shows that $f$ is (strongly) local if and only if $R\to R/\ker(f)$ and $R/\ker(f)\to S$ are (strongly) local. 
\end{remark}

\begin{corollary}\label{3.10} Let $R\subseteq S\subseteq T$ be a tower of extensions.  
\begin{enumerate}
\item $R\subseteq T$ is SL if and only if $R\subseteq S$ and $S\subseteq T$ are SL.
\item If $R\subseteq S$ and $S \subseteq T$ are  local, then $R\subseteq  T$ is  local.
\item If $R\subseteq  T$ is  local, then $R\subseteq  S$ is  local.
\end{enumerate}
\end{corollary}
\begin{proof} Obvious by Proposition \ref{3.1}.
\end{proof}

\begin{proposition}\label{3.2} If $\{R_i\subseteq S_i\}_{i\in I}$ is a family of extensions, then $\Pi_{i\in I}R_i\subseteq\Pi_{i\in I}S_i$ is a (strongly) local extension if and only if all the elements of the family are (strongly) local.
\end{proposition}

\begin{proposition}\label{3.21} Consider a pullback square in the category of commutative rings:
$$\begin{matrix}
        R         & \to &        S          \\
\downarrow &  {}  & \downarrow \\
       R'         & \to &     S'         
\end{matrix}$$
 whose horizontal maps are extensions. Then if $R'\subseteq S'$ is (strongly) local, so is $R\subseteq S$.
\end{proposition}
\begin{proof} Assume first that $R'\subseteq S'$ is local, so that $\mathrm {U}(R')=\mathrm{U}(S')\cap R'$. As $\mathrm{U}(R)\subseteq\mathrm{U}(S)\cap R$, it is enough to show that $\mathrm{U}(S)\cap R\subseteq\mathrm{U}(R)$. Let $f:S\to S'$ and $a\in\mathrm{U}(S)\cap R$. Then, $ b:=f(a)\in\mathrm{U}(S')\cap R'=\mathrm{U}(R')$ is a unit in $R'$. Let $g: R\to R'$ and consider $x=(a,b)=(a,f(a))\in R$ with $b=g(x)=f(a)$. Set $a':= a^{-1}\in\mathrm{U}(S)$ and $b':=f(a')$. It follows that $b'\in\mathrm{U}(S')$ satisfies $bb '=1$ in $S'$ and is the (unique) inverse of $b$, so that $b'\in R'$. Set $x':=(a',b')$, with $ b'=g(x')=f(a')$. Then, $xx'=(aa',bb')=1$, so that $x\in\mathrm{U}(R)$. It follows that for any $a\in\mathrm{U}(S)\cap R$, there exists a unique $x=(a,f(a))\in\mathrm{U}(R)$, 
from which we can infer that
 $\mathrm{U}(S)\cap R\subseteq\mathrm{U}(R)$ and $R\subseteq S$ is local.

Assume now that $R'\subseteq S'$ is SL, so that $\mathrm{U}(R')=\mathrm{U}(S')$. The proof is similar, taking $a\in\mathrm{U}(S)$ instead of $a\in\mathrm{U}(S)\cap R$ and taking in account that $\mathrm{U}(S')=\mathrm{U}(R') $. Then   $R\subseteq  S$ is  SL.
\end{proof}

 \begin{corollary}\label{5.1} Let $R\subseteq S$ be an extension sharing an ideal $I$ such that $R/I\subseteq S/I$ is (strongly) local. Then $R\subseteq S$ is (strongly) local.
 \end{corollary}
\begin{proof} Obvious by Proposition \ref{3.21}.
\end{proof}

\begin{proposition}\label{3.22} Let $\{R\subseteq S_i\}_{i\in I}$ be an upward directed family of (strongly) local extensions. Then so is $R\subseteq\cup[S_i\mid i\in I]$.
\end{proposition}
\begin{proof} Set $\mathcal F:=\{S_i\}_{i\in I}$ and $T:=\cup[S_i\mid i\in I]$. Let $x\in\mathrm U(T)$. Since $\mathcal F$ is an upward directed family of extensions of $R$,   there exists $y\in\mathrm U(T)$ such that $xy=1\ (*)$ and there exists some $i\in I$ such that $x,y\in S_i$, with $xy=1$ in $S_i$ by $(*)$. This shows that $x\in\mathrm U(S_i)\ (**)$.

If $R\subseteq S_i$ is local for each $S_i$, let $x\in\mathrm U(T)\cap R$. Then, $(**)$ shows that $x\in\mathrm U(S_i)\cap R=\mathrm U(R)$ giving $\mathrm U(T)\cap R=\mathrm U(R)$ and $R\subseteq T$ is  local.

If $R\subseteq S_i$ is SL for each $S_i$, let $x\in\mathrm U(T)$. Then, $(**)$ shows that $x\in\mathrm U(S_i)=\mathrm U(R)$ giving $\mathrm U(T)=\mathrm U(R)$ and $R\subseteq T$ is SL.
\end{proof}

\begin{proposition}\label{3.24} Let $R\subseteq S$ be a ring extension such that $R_M\subseteq S_M$ is (strongly) local for any $M\in\mathrm{MSupp}(S/R)$, then so is $R\subseteq S$.
\end{proposition}
\begin{proof} Assume first that $R_M\subseteq S_M$ is local for any $M\in\mathrm{MSupp}(S/R)$. Since we always have $\mathrm{U}(R)\subseteq\mathrm{U}(S)\cap R$, it is enough to prove that any $x\in\mathrm{U}(S)\cap R$ is in $\mathrm{U}(R)$. So, let  $x\in\mathrm{U}(S)\cap R$. There exists $y\in\mathrm{U}(S)$ such that $xy=1\ (*)$ in $S$,  so that $(x/1)(y/1)=1\ (**)$ in $S_M$ for any $M\in\mathrm{MSupp}(S/R)$. Then, $x/1\in\mathrm{U}(S_M)\cap R_M=\mathrm{U}(R_M)$ for any $M\in\mathrm{MSupp}(S/R)$. It follows that for any $M\in\mathrm {MSupp}(S/R)$, there exists some $y_M/s_M\in\mathrm{U}(R_M)$ such that $(x/1)(y_M/s_M)=1$ in $R_M\subseteq S_M$. This implies that by $(**)$ we get $y_M/s_M=y/1$ in $S_M$ by the uniqueness of the inverse, so that $y/1\in R_M$ for any $M\in\mathrm{MSupp}(S/R)$. Moreover, let $M\not\in\mathrm{MSupp}(S/R)$. Then, $R_M=S_M$, 
so  that
 $y/1\in R_M$ for any $M\in\mathrm{Max}(R)$ and $(*)$ shows that $x\in\mathrm{U}(R)$.

Assume that $R_M\subseteq S_M$ is SL for any $M\in\mathrm{MSupp}(S/R)$. Since we always have $\mathrm{U}(R)\subseteq\mathrm{U}(S)$, it is enough to prove that any $x\in\mathrm{U}(S)$ is in $\mathrm{U}(R)$. So, let $x\in\mathrm{U}(S)$. There exists $y\in\mathrm{U}(S)$ such that $xy=1\ (*)$ in $S$ 
 which entails that
 $(x/1)(y/1)=1\ (**)$ in $S_M$ for any $M\in\mathrm{MSupp}(S/R)$. For the rest of the proof, it is enough to copy the same part of the proof of the local case. 
\end{proof}

\begin{proposition}\label{3.25} Let $R\subseteq S$ be a ring extension and $\Sigma$ a saturated multiplicative closed subset of $R$ which is also a saturated multiplicative  closed subset of $S$. If $R_{\Sigma}\subseteq S_{\Sigma}$ is (strongly) local, then so is $R\subseteq S$.
\end{proposition}
\begin{proof} Assume first that $R_{\Sigma}\subseteq S_{\Sigma}$ is local. Since we always have $\mathrm{U}(R)\subseteq\mathrm{U}(S)\cap R$, it is enough to prove that any $x\in\mathrm{U}(S)\cap R$ is in $\mathrm{U}(R)$. So, let $x\in\mathrm{U}(S)\cap R$. There exists $y\in\mathrm{U}(S)$ such that $xy=1\ (*)$ in $S$ which implies that $(x/1)(y/1) =1\ (**)$ in $S_{\Sigma}$, which is equivalent to $u=uxy$ for some $u\in\Sigma$. In particular, $y\in\Sigma\subseteq R$ because $\Sigma$ is closed in $S$. Using $(*)$, we get that $x\in\mathrm{U}(R)$ and $R\subseteq S$ is  local. 

Assume now that $R_{\Sigma}\subseteq S_{\Sigma}$ is SL. Since we always have $\mathrm{U}(R)\subseteq\mathrm{U}(S)$, it is enough to prove that any $x\in\mathrm{U}(S)$ is in $\mathrm{U}(R)$. So, let $x\in\mathrm{U}(S)$. There exists $y\in\mathrm{U}(S)$ such that $xy=1\ (*)$ in $S$, whence $(x/1)(y/1)=1\ (**)$ in $S_{\Sigma}$. For the rest of the proof, it is enough to copy the same part of the proof of the local case. 
\end{proof}

 We consider the Nagata idealization $R (+) M$ of an $R$-module $M$.

\begin{proposition}\label{3.61} (1) Let $M$ be an $R$-module and $N$ an $ R$-submodule of $M$. Then, $R\subseteq R(+)M$ and $R(+)N\subseteq R(+)M$ are local extensions. But $R\subseteq R(+)M$ is SL if and only if $M=0$ and $R(+)N\subseteq R(+)M$ is SL if and only if $M=N$.

(2) If $R\subseteq S$ is a ring extension and $M$ an $S$-module, then $M$ is also an $R$-module and $R\subseteq S$ is SL if and only if $R(+)M\subseteq S(+)M$ is SL. 
\end{proposition} 
\begin{proof} (1) Let $(x,m)\in R(+)M$. Then, $(x,m)\in\mathrm U(R(+)M)$ if and only if $ x\in\mathrm U(R)$. Under this condition, we have $(x,m)^{-1}=(x^{-1},-x^{-2}m)$. 
It follows that $\mathrm U(R(+)M)\cap R=\mathrm U(R)$, so that $R\subseteq R(+)M$ is local but $R\subseteq R(+)M$ is SL if and only if $M=0$. 

If $N$ a submodule of $M$ we get that $\mathrm U(R(+)M)\cap(R(+)N)=\mathrm U(R(+)N)$, whence $R(+)N\subseteq R(+)M$ is local but $R(+)N\subseteq R(+)M$ is SL if and only if $M=N$. 

(2) If $M$ an $S$-module, then obviously, $M$ is also an $R$-module. Let $(x,m)\in S(+) M$. By the proof of (1), $(x,m)\in\mathrm U(S(+)M)$ if and only if $x\in\mathrm U(S)$. A similar equivalence holds for $\mathrm U(R(+)M) $ and $\mathrm U(R)$. Then the equivalence of (2) is obvious.
\end{proof}

We recall in the next section some results of local extensions we get in \cite{Pic 17} and add some new results. We will look more precisely at SL extensions in the other  sections.

\section{Properties of local extensions}

An extension $R\subseteq S$ is called {\it survival} if for each ideal $I$ of $R$  such that $I\neq R$, then $IS\neq S$ (equivalently $PS\neq S$ for each $P\in\mathrm{Spec}(R)$). 

\begin{proposition}\label{3.3}   Lying-over or survival extensions  are local.

In case the ideals of a ring $R$ are linearly ordered, an extension $R\subseteq S$ is survival if and only if it is  local.
\end{proposition}
\begin{proof} \cite[Definition before Proposition 8.8]{Pic 17} gives the result for a survival extension. 

Assume that $R\subseteq S$ has lying-over. Since we always have $\mathrm{U}(R)\subseteq\mathrm{U}(S)\cap R$, it is enough to prove that any $x\in\mathrm{U}(S)\cap R$ is in $\mathrm{U}(R)$. So, let $x\in\mathrm{U}(S)\cap R$ be such that $x\not\in\mathrm{U}(R)$. There exists some $P\in\mathrm{Spec}(R)$ such that $x\in P$ and there exists some $Q\in\mathrm{Spec}(S)$ lying over $P$. Then $x\in P\subseteq PS\subseteq Q$, a contradiction with $x\in\mathrm{U}(S)$. 

Assume that the ideals of a ring $R$ are linearly ordered and that the extension $R\subseteq S$ is local. Let $I$ be an ideal of $R$ such that $I\neq R$ and $IS=S$. There exists some $x_1,\ldots,x_n\in I$ and $s_1, \ldots,s_n\in S$ such that $\sum_{i=1}^ns_ix_ i=1\ (*)$. But the $Rx_i$ are linearly ordered. Let $x_k$ be such that $Rx_i\subseteq R x_k$ for each $i\in\mathbb N_n$. Then $(*)$ implies $sx_k=1$ for some $s\in S$, so that $x_k\in\mathrm{U}(S)\cap R=\mathrm{U}( R)$, a contradiction. Then, $R\subseteq S$ is survival. 
\end{proof}

\begin{proposition}\label{3.4} \cite[Proposition 8.8]{Pic 17} An extension $R\subseteq S$ is survival if and only if  $R(X) \subseteq S(X)$ is local. 
\end{proposition}

Recall that an extension $R\subseteq S$ is called {\it Pr\"ufer} if $R\subseteq T$ is a flat epimorphism for each $T\in[R,S]$ (or equivalently, if $R\subseteq S$ is a normal pair) \cite[Theorem 5.2, page 47]{KZ}.

In \cite{Pic 5}, we defined an extension $R\subseteq S$ to be {\it quasi-Pr\"ufer} if it can be factored $R\subseteq R'\subseteq S$, where $R\subseteq R'$ is integral and $R'\subseteq S$ is Pr\"ufer. An FCP extension is quasi-Pr\"ufer \cite[Corollary 3.4]{Pic 5}. 
  
\begin{definition} \label{Morita} A ring extension $R\subseteq S$ has a greatest flat epimorphic subextension $R\subseteq\widehat R$ that we call the {\it Morita hull} of $R$ in $S$ \cite[Corollary 3.4]{M}. We say that the extension is {\it Morita-closed} if $\widehat R = R$.
\end{definition}
In fact, $\widehat R$ coincide with the weakly surjective hull $\mathrm{M}(R,S)$ of \cite{KZ}, since weakly surjective morphisms $f:R\to S$ are characterized by $R/\ker f\to S$ is a flat epimorphism \cite[Proposition p.3]{O}. Our terminology is justified by the fact that the Morita construction is earlier. 
 The Morita hull can be computed by using a (transfinite) induction \cite{M} as follows.
Let $S'$ be the set of all $s\in S$, such that there is some ideal $I$ of $R$, such that $IS= S$ and $Is \subseteq R$, or equivalently, $ S' = \{s\in S\mid S =(R:_Rs)S\}$. 
Then $R\subseteq S'$ is a subextension of $R\subseteq S$. We set $S_1:=S'$ and $S_{i+ 1}:= (S_i)' \subseteq S_i$, which defines the transfinite induction. 
 
We introduce the condition ($\star$) on a ring extension $R\subseteq S$: an element $s\in S$ belongs to $R$ if (and only if) $S=(R:_Rs)S$. The condition ($\star$) is clearly equivalent to the Morita-closedness of the extension.
 
An integral extension is Morita-closed because an injective flat epimorphism is trivial as soon as it has LO \cite[Lemme 1.2, p.109]{L}.
 
An extension $R\subseteq S$ is Morita-closed if $R$ is zero-dimensional because a prime ideal of $R$ can be lifted up to $S$, since such a prime ideal is minimal and then it is enough to again apply \cite[Lemme 1.2, p.109]{L}.

\begin{proposition}\label{3.6} A Morita-closed extension $R\subseteq S$ is local. Therefore the Morita hull $\widehat R$ of an extension $R\subseteq S$ is such that $\widehat R \subseteq S$ is local.
\end{proposition} 
\begin{proof} Use the multiplicatively closed subset $\Sigma_S:=\{r\in R\mid r\in\mathrm U(S)\}$ and the factorization $R\subseteq R_{\Sigma_S} \subseteq S$.
\end{proof}
 
The condition $(\star)$ defined after Definition \ref{Morita} is stronger than the local property. A flat epimorphic extension does not verify the condition $(\star )$ because of \cite[Exercise 8, p.242]{STEN}. Actually, a flat epimorphic extension does not need to be local: it is enough to consider a localization $R\to R_P$, where $P$ is a prime ideal of an integral domain $R$.
  
We use the results of Olivier about pure extensions \cite{OP}. Recall that an injective ring morphism $f:R\to S$ is called {\it pure} if $R'\to R'\otimes_RS$ is injective for each ring morphism $R\to R'$, whence purity is an universal property. A faithfully flat morphism is pure.

A pure ring morphism is a strict monomorphism of the category of commutative rings \cite[Corollaire 5.2, page 21]{OP}. This last condition can be characterized in the category of commutative unital rings as follows. If $f:R\to S$ is a ring morphism, the dominion of $f$ is $\mathrm{D}(f):=\{s\in S\mid s\otimes 1=1\otimes s\ \mathrm{in}\ S\otimes_RS\}$. 

A ring extension $R\subseteq S$ is a {\it strict} monomorphism if and only if its dominion is $R$. Now a ring morphism is an epimorphism if and only if its dominion is $S$ \cite[Lemme 1.0, page 108]{L}. It follows that a ring extension is strict and an epimorphism if and only if it is trivial. 
 
Denote by $D$ the dominion of an extension $R\subseteq S$ then $D\subseteq S$ is strict, because $S\otimes_DS =S\otimes_RS$.

We will consider minimal (ring) extensions, a concept that was introduced by Ferrand-Olivier \cite{FO}. Recall that an extension $R\subset S$ is called {\it minimal} if $[R, S]=\{R,S\}$. A minimal extension is either a flat epimorphism or a strict monomorphism, in which case it is finite. A minimal integral extension is strict \cite[Th\'eor\`eme 2.2(ii)]{FO}. 

\begin{lemma}\label{STRICTINV} A strict monomorphism $f: R\to S$ is local (e.g. either pure or minimal integral). 
\end{lemma}
\begin{proof} If $a\in R$ is a unit in $S$, there is a factorization $R\to R_a\to S$. From the natural map $R_a\otimes_RR_a\to S\otimes_RS$ and $\mathrm{D}(R\to R_a )=R_a$, because $R\to R_a$ is an epimorphism, we get that $R_a\subseteq\mathrm{D}(f)=R$. 
\end{proof}

 We define the class $\mathcal{TP}$ of rings $R$ such that $\mathrm{Pic}(R) =0$. It contains semi-local rings and Nagata rings. 

\begin{proposition}\label{3.18} A Pr\"ufer local extension $R\subseteq S$ over a $\mathcal{TP}$ ring  is trivial. 
\end{proposition} 

\begin{proof} Assume that $R\neq S$. Since $R\subset S$ is Pr\"ufer, it is a flat epimorphism, so that there exists $M\in\mathrm{Max}(R)$ such that $MS=S$ and $1=\sum_{i=1}^nm_is_i$, for some positive integer $n,\ m_i\in M$ and $s_i\in S$ for each $i\in\mathbb N_n$. Set $I:=\sum_{i=1}^nR m_i$. Then $IS=S$, so that $I$ is an $S$-regular $R$-submodule of $S$ finitely generated. According to \cite[Theorem 1.13, p.91]{KZ}, $I$ is $S$-invertible, and then a projective $R$-module of rank 1 by \cite[Lemma 4.1, p.109]{KZ}. It follows that $I$ is a free $R$-module
 which implies that $I$ is a principal ideal of $R$.
 
Set $I:=Rx$, with $x\in R$. But $IS=S$ gives $xS=S$, and $x$ is a unit in $S$. But $x\in R$ implies $x\in\mathrm U(S)\cap R=\mathrm U(R)$, a contradiction with $x\in I\subseteq M$. Then, $R=S$. 
\end{proof}

This Proposition generalizes \cite[Proposition 8.9]{Pic 17} which holds for an extension over an arithmetical ring. 

\section{Generalities about SL extensions}

We first give examples of SL  extensions.
\begin{example} \label{exam} 
(1a) If $R\subseteq S$ is an extension we can consider $L (S):=R[\mathrm U(S)]$. Then $L(S)\subseteq S$ is SL and $L(S)$ is the smallest element $T$ of $[R,S]$ such that $T\subseteq S$ is SL. Actually $L(S)$ is the intersection of all $T\in[R,S]$ such that $T\subseteq S$ is SL.

(1b) There exist also subextensions $R\subseteq U$ of $[R,S]$, maximal with respect to the property SL. We will call them {\it MSL}-subextensions. The proof uses Zorn's Lemma.

If $R\subseteq S$ is a chained extension, then $U:=\cup[V\in [R,S]\mid\mathrm U(R)=\mathrm U(V)]$ is 	{\it the} MSL-subextension.

In Theorem \ref{5.05}, we prove that when $R\subseteq S$ is an FCP  extension and $S$ is semilocal, there exists a unique MSL-subextension.  

(2) Let $R\subseteq S$ be an extension where $S$ is a Boolean ring, then the extension is SL. This is obvious because the only unit of $S$ and $R$ is $1$ since $R$ is also a Boolean ring.

(3) Let $R$ be a reduced ring. Then $R\subset R[X]$ is SL because an invertible polynomial of $R[X]$ is of the form $a+Xf(X)$ where $a$ is invertible in $R$ and the coefficients of $f(X)$ are nilpotent, so that $f(X)=0$.

On the other hand, $R\subset R[[X]]$ is never SL because any power series with a unit of $R$ as first coefficient is a unit in $R[[X]]$.

(4) Let $\Sigma$ be a multiplicative closed subset of a ring $R$ such that $R\subseteq R_{\Sigma}$ is SL. Then $\Sigma=\mathrm U(R)$ and $R=R_{\Sigma}$. Of course, $\mathrm U(R)\subseteq\Sigma$. Assume $\mathrm U(R)\neq\Sigma$ and let $x\in\Sigma\setminus\mathrm U(R)$. Then, $x\in\mathrm U(R_{\Sigma})=\mathrm U(R)$, a contradiction, so that $\Sigma=\mathrm U(R)$ and $R=R_{\Sigma}$. 

(5) The ring morphism $j:R\to R/\mathrm J(R)$ is SL. Indeed if $\bar x$ is the class of an element $x\in R$ that is a unit in $R/\mathrm J(R)$, there is some $y\in R$, such that $xy-1\in\mathrm J(R)$. Then $1-(1-xy)$ is a unit in $R$; so that $x\in \mathrm U(R)$. 

(6) Let $R\subseteq S$ be an extension such that $1\neq -1$ and $\mathrm U(S)$ is a simple group. Since $\mathrm U(R)$ is a subgroup of $\mathrm U(S)$, the only possibilities are either $\mathrm U(R)=\{1\}\ (*)$ or $\mathrm U(R)=\mathrm U(S)\ (**)$. But $1\neq - 1$ and $-1\in\mathrm U(R)$ show that only case $(**)$ can occur. Then $R\subseteq S$ is SL. 

(7) Let $R\subseteq S$ be an extension where $1\neq -1$ and $\mathrm U(S)$ has a finite prime order. Then $\mathrm U(S)$ is a simple group and $R\subseteq S$ is SL by (6).

We will see that in the FCP case, the situation  $1\neq -1$ often occurs through the ring $\mathbb Z/2\mathbb Z$.

(8) An SL extension $R\subseteq S$  is trivial  if  $(S,M)$ is a local ring.  Indeed since $S=\mathrm U(S)\cup M=\mathrm U(R)\cup M$ and any $x\in M$ is such that $1+x\in \mathrm U(S)=\mathrm U(R)\subseteq R$, we get that $M\subseteq R$, giving $R=S$. 
\end{example}
 
\begin{proposition}\label{3.01} The following holds for an extension $R\subseteq S$ where $R$ is reduced: 
\begin{enumerate}
\item If $s\in S$ is not algebraic over $R$, then $R\subseteq R[s]$ is SL.
\item In case $S$ is also reduced and  $V\in[R,S]$ is  a MSL-subextension, then  $V\subseteq S$ is algebraic.
\end{enumerate}    
 \end{proposition} 
\begin{proof} (1) Let $s\in S$ which is not algebraic over $R$ and consider the morphism $\varphi:R[X]\to R[s]$ defined by $\varphi(X)=s$. Then $\varphi$ is a surjective morphism which is also injective since $s$ is not algebraic over $R$. Then $R[X]\cong R[s]$ so that $R\subseteq R[s]$ is SL by Example \ref{exam}(3).

(2) Assume, moreover, that $S$ is also reduced. By Example \ref{exam}(1), there exists an {\it MSL}-subextension $ V\in[R,S]$.  
Then, $V$ is also reduced. Assume that $S\neq V$ and let $s\in S\setminus V$. If $s$ is not algebraic over $V$, then $V\subseteq V[s]$ is SL by (1) and so is $R\subseteq V[s]$ because $\mathrm U(R)=\mathrm U(V)=\mathrm U(V[s])$, a contradiction with the maximality of $V$. Then, any $s\in S\setminus V$ is algebraic over $V$ and then $V\subseteq S$ is algebraic.
\end{proof}

\begin{corollary}\label{3.02} Let $R\subseteq S$ be an extension where $S$ is reduced. There exists $U\in[R,S]$ such that $R\subseteq U$ is SL and $U\subseteq S$ is quasi-Pr\"ufer.   
 \end{corollary} 
\begin{proof} By Example \ref{exam}(1), there exists an MSL-subextension $U\in[R,S]$. Moreover, $U\subseteq S$ is algebraic as $U\subseteq V$ for any $V\in[U,S]$ by Proposition \ref{3.01}. This shows that $U\subseteq S$ is a residually algebraic pair, and then is quasi-Pr\"ufer by \cite[Theorem 2.3]{Pic 5}.
\end{proof}

\begin{proposition}\label{3.28} Let $R\subseteq S$ be a ring extension. The following conditions are equivalent:
\begin{enumerate}
\item $R\subseteq S$ is SL.
\item $R[X]\subseteq S[X]$ is SL.
\item $R+XS[X]\subseteq S[X]$ is SL.
\end{enumerate}
If these conditions hold, then $\mathrm{Nil}(R)=\mathrm{Nil}(S)$.
\end{proposition} 

\begin{proof}  We begin to remark that $1-x$ is a unit if $x$ is nilpotent for any $x\in R$ (resp. $x\in S$). Then, $\mathrm U(R)=\mathrm U(S)\Rightarrow\mathrm{Nil}(R)=\mathrm{Nil}(S)$.

(1) $\Leftrightarrow$ (2) We know that $\mathrm U(R[X])=\mathrm U(R)+X\mathrm{Nil}(R)[X]\ (*)$ and $\mathrm U(S[X])=\mathrm U(S)+X\mathrm{Nil}(S)[X]\ (**)$. Then, we get the equivalence applying the previous remark to $R\subseteq S$ and $R[X]\subseteq S[X]$ and using $(*)$ and $(**)$.

(2) $\Rightarrow$ (3) since $R[X]\subseteq S[X]$ SL implies by Corollary \ref{3.10} that $R+XS[X]\subseteq S[X]$ is SL because $R+XS[X]\in[R[X],S[X]]$.

(3) $\Rightarrow$ (1) Let $a\in\mathrm U(S)\subseteq\mathrm U(S[X])=\mathrm U(R+XS [X])$, which gives $a\in S\cap\mathrm U(R+XS[X])\subseteq R+XS[X]$. This shows that $ a\in R$. The same property holds for $b:=a^{-1}\in S$, so that $a\in\mathrm U(R)$ and $R\subseteq S$ is SL.
\end{proof}

\begin{corollary}\label{3.29} Let $R\subseteq S$ be a ring extension such that $S$ is reduced. Then $R\subseteq R+XS[X]$ is SL.
\end{corollary} 

\begin{proof} Let $a\in\mathrm U(R+XS[X])\subset\mathrm U(S[X])\cap(R+XS[X])$, so that $a=c+Xf(X)$, where $c\in R$ and $f(X)\in S[X]$. But, as $a$ is an element of $\mathrm U (S[X])$, we have $f(X)\in\mathrm{Nil}(S)[X]=0$ because $S$ is reduced. Then, $a=c\in R\cap\mathrm U(S[X])\subseteq R\cap\mathrm U(S)$. The same property holds for $b:= a^{-1}\in\mathrm U(R+XS[X])$, so that $a\in\mathrm U(R)$ and $R\subseteq R+XS[X]$ is SL.
\end{proof}

\begin{proposition}\label{3.27} Let $R\subseteq S$ be a ring extension. If $R(X)\subseteq S(X)$ (resp. $R[[X]]\subseteq S[[X]]$) is SL, then $R= S$.
\end{proposition} 

\begin{proof} Assume that $R(X)\subseteq S(X)$ is SL and let $s\in S$. Then, $P(X):=s+ X\in S[X]$ is such that $P(X)/1\in\mathrm U(S(X))=\mathrm U(R(X))$. This implies that there exist $f(X),g(X)\in R[X]$ with content equal to $R$ such that $P(X)/1=f(X)/g(X)$. Set $f(X):=\sum_{i=0}^{n+1}a_iX^{i}$ and $g(X):=\sum_{i=0}^nb_iX^{i},\ a_i,b_i\in R$ such that $g(X)P(X)=f(X)$. It follows that $sb_0=a_0,b_n=a_{n+1}$ and $sb_i+b_{i-1}= a_i $ for any $i\in\mathbb N_n\ (*)$. Since $c(g)=R$, there exist $\lambda_0,\ldots,\lambda_n\in R$ such that $\sum_{i=0}^n\lambda_ib_i=1$. Multiplying each equality of rank $i$ of $(*)$ by $\lambda_i$  and adding each of these equalities for each $i\in\{0, \ldots,n\}$, we get $s(\sum_{i=0}^n\lambda_ib_i)=s=\lambda_0a_0+\sum_{i=1}^n\lambda_i(a_i-b_{i-1})\in R$, so that $S=R$.

Assume now that $R[[X]]\subseteq S[[X]]$ is SL. We know that $\mathrm U(R[[X]])=\mathrm U(R)+XR[[X]]$ and $\mathrm U(S[[X]])=\mathrm U(S)+XS[[X]]$. It follows that  $\mathrm U(R)+XR[[X]]=\mathrm U(S)+XS[[X]]$. Let $s\in S$. Then $1+sX\in \mathrm U(S)+XS[[X]]=\mathrm U(R)+XR[[X]]$, so that $s\in R$ and $R=S$. 
\end{proof}

 \begin{definition}\label{4.0}  Let $R$ be a ring. 
 \begin{enumerate}
\item A polynomial $p(X)\in R[X]$ is called {\it comonic} if $p(0)\in \mathrm U(R)$.
\item A ring extension $R\subseteq S$ is called {\it co-integrally closed} if any $x\in S$ which is a zero of a comonic polynomial of $R[X]$ is in $R$.
\end{enumerate} 
 \end{definition} 
 
 \begin{proposition}\label{4.1} Let $R\subseteq S$ be a ring extension. The following statements hold:
\begin{enumerate}
\item If $R\subseteq S$ is SL, then $R\subseteq S$ is  co-integrally closed.
\item Let $R\subseteq S$ be an integral ring extension. Then $R\subseteq S$ is  SL if and only if  $R\subseteq S$ is  co-integrally closed.
\end{enumerate}
\end{proposition}
\begin{proof} (1) Assume that $R\subseteq S$ is SL and let $x\in S$ be a zero of a comonic polynomial $p(X):=\sum_{i=0}^na_iX^{i}\in R[X]$. Then, $p(0)=a_0\in\mathrm U(R)$ and $\sum_{i=0}^na_ix^{i}=0$, so that $x(\sum_{i=1}^na_ix^{i-1})=-a_0\in\mathrm U(R)=\mathrm U(S)$ shows that $x\in\mathrm U(S)=\mathrm U(R)\subseteq R$. Then,  $R\subseteq S$ is  co-integrally closed.

(2) One part of the proof is gotten in (1). So, assume that $R\subseteq S$ is  co-integrally closed. Obviously, $\mathrm U(R)\subseteq\mathrm U(S)$. Let $x\in\mathrm U (S)$. Since $R\subseteq S$ is an integral ring extension, there exists a monic polynomial $p(X):=\sum_{i=0}^na_iX^{i}\in R[X]$ with $a_n=1$, such that $p(x)=0$, giving $x^n+\sum_{i=0}^{n-1}a_ix^{i}=0\ (*)$. Since $x\in\mathrm U(S)$, multiplying $(*)$ by $x^{-n}$, we get $1+\sum_{i=0}^{n-1}a_ix^{i-n}=1+\sum_{i=0}^{n-1}a_i(x^{-1})^{n-i}=0$ which shows that $x^{-1}$ is a zero of the comonic polynomial $q(X):=1+\sum_{i=0}^{n-1}a_iX^ {n-i}=\sum_{i=1}^na_{n-i}X^{i}+1\in R[X]$. Then, $x^{-1}\in R$. To sum up, we have shown that any $x\in\mathrm U(S)$ is such that $x^{-1}\in R$. Setting $y:=x^{-1}$, which is also in $\mathrm U(S)$, the previous proof gives that $y^{-1}=x$ is in $R$. Moreover, since it also shows that $x$ and $x^{-1}$ are in $R$, this implies that $x\in\mathrm U(R)$. To conclude, $\mathrm U(R)=\mathrm U(S)$ and $R\subseteq S$ is SL.
\end{proof}
 
 \section{$\mathrm J$-regular rings}
 
In this section, we look at properties of $\mathrm J$-regular rings, which will play an important role in the following study of SL extension.  An absolutely flat ring is in this paper called a (Von Neumann) {\it regular ring.}
Actually, many results need only rings $R$ with a Jacobson radical $J$ such that $R/J$ is regular. They  are called {\it $\mathrm J$-regular} in the literature.  

We recall some results concerning regular rings.
 
(1) A ring $R$ is regular if for any $x\in R$, there exists $y\in R$ such that $x^2 y=x$. Under these conditions, such an $y$ is unique when satisfying $y^2x=y$ \cite[Lemme, p.69]{OP1}. Moreover, setting $e:=xy$ and $u:=1-e+x$, we get that $e$ is an idempotent, $u$ is a unit with $1-e+y$ as inverse and $x=eu$. 

(2) A ring is regular if and only if it is reduced and zero-dimensional.

(3) Let $R$ be a regular ring. For any $P\in\mathrm{Spec}(R)$,  there is an  isomorphism $R/P\cong R_P$.

\begin{lemma}\label{STRICTINV1} If $f:R\to S$ is a strict monomorphism and $S$ is regular, then $R$ is regular.
\end{lemma}
\begin{proof} See \cite[diagram page 40 and Proposition 19]{PICANA}.
\end{proof}

For any ring $R$ there is an (Olivier) ring epimorphism $t:R\to\mathcal O(R)$, whose spectral map is bijective, such that $\mathcal O(R)$ is regular and such that any ring morphism $R \to S$ where $S$ is regular can be  factored $R\to\mathcal O(R)\to S$. 

Note that $\ker(t)=\mathrm{Nil}(R)$.

This property is a consequence of the following facts: $\mathrm{V}(t^{-1}(I))=\overline{{}^ {a}t(\mathrm{V}(I))}$ for any ideal $I$ of $\mathcal O(R)$ \cite[Proposition 1.2.2.3, p.196]{EGA} and $\mathcal O(R)$ is reduced. 

The existence of the ring $\mathcal O(R)$, called {\it the universal regular (absolutely flat) ring associated to} $R$ is due to Olivier \cite[Proposition, p.70]{OP1}.

\begin{lemma}\label{5.22} If $f:R\to S$ is an epimorphism and $M\in\mathrm{Max}(R)$, there exists $N\in\mathrm{Max}(S)$ such that $M={}^{a}f(N)$ if and only if $N=f(M)S$. 

If in addition ${}^{a}f$  is bijective, then, for any $M\in\mathrm{Max}(R)$, we have $f(M)S\in \mathrm{Max}(S)$ and $M={}^{a}f[f(M)S]$.
\end{lemma}
\begin{proof} Consider the following commutative diagram:
$$\begin{matrix}
        R         &           \overset{f}\to      &        S          \\
\downarrow &                   {}                & \downarrow \\
     R/M        & \overset{\overline f}\to &     S/f(M)S         
\end{matrix}$$
where $\overline f$ is deduced from $f$. Since $f$ is an epimorphism, so is $\overline f$. Assume that $M:={}^{a}f(N)$ for some $N\in\mathrm{Max}(S)$. Since $f(M)S\subseteq N\subset S$, it follows that $M={}^{a}f[f(M)S]$, so that $\overline f$ is injective. Now, $R/M$ being a field, $\overline f$ is surjective, and then an isomorphism. Hence, $S/f(M)S$ is a field and $f(M)S\in \mathrm {Max}(S) $, which infers that $N=f(M)S$. 

Conversely, assume that $N=f(M)S$ for some $N\in\mathrm{Max}(S)$. Then, $f(M)S\subset S$. It follows that ${}^{a}f[f(M)S]\in\mathrm{Spec}(R)$ with $M\subseteq{}^{a}f       [f(M)S]$

\noindent$\subset R$.  But $M\in\mathrm{Max}(R)$ implies that $M={}^{a}f[f(M)S]={}^{a}f (N)$. 

Now if ${}^{a}f$ bijective,  for any $M\in\mathrm{Max}(R)$, there exists $N\in\mathrm{Max}(S)$ such that $M={}^{a}f(N)$. Then the first part of the Lemma gives that $N:=f(M)S\in\mathrm{Max}(S)$ and $M={}^{a}f[f(M)S]$.
\end{proof}

 \begin{proposition}\label{5.23}Let $t:R\to\mathcal O(R)$ be the  Olivier ring epimorphism, $\mathcal T:=\{N\in\mathrm{Spec}(\mathcal O(R))\mid{}^{a}t(N)\in\mathrm {Max}(R)\}$ and $K:=\cap[N\in\mathcal T]$. Then, 
 ${}^{a}t(\mathcal T)=\mathrm{Max}(R)$ and any $N\in\mathcal T$ is of the form $t(M)\mathcal O(R)$ where $M\in \mathrm{Max}(R)$.

If $R$ is $\mathrm J$-regular, there is an isomorphism $R/\mathrm J(R)\cong \mathcal O(R)/K$.
\end{proposition}
\begin{proof} 
Since ${}^{a}t(N)\in\mathrm{Max}(R)$ for any $N\in\mathcal T$, we have ${}^{a}t(\mathcal T)\subseteq\mathrm{Max}(R)$. Now, let $M\in\mathrm{Max}(R)$. There exists $N\in\mathrm{Spec}(\mathcal O(R))$ such that $M={}^{a}t(N)$ because ${}^{a}t$ is bijective, so that $N\in\mathcal T$, giving $M\in{}^{a}t(\mathcal T)$. To end, ${}^{a}t(\mathcal T)=\mathrm{Max}(R)$.

Let $N\in\mathcal T$ and set $M:={}^{a}t(N)\in\mathrm{Max}(R)$. By Lemma \ref{5.22}, $N=t(M)\mathcal O(R)$.

Assume, moreover, that $R$ is $\mathrm J$-regular. Then, $R/\mathrm J(R)$ is regular, that is zero-dimensional. Setting $K:=\cap[N\in\mathcal T]$, we get $t^{-1}(K)=t^{-1}(\cap[N\in\mathcal T])=\cap[{}^{a}t(N)\mid N\in\mathcal T]=\cap[M\in{}^{a}t(\mathcal T)]=\cap[M\in \mathrm{Max}(R)]=\mathrm J(R)$. Consider the following commutative diagram:
$$\begin{matrix}
          R             &           \overset{t}\to      &   \mathcal O(R) \\
  \downarrow     &                   {}                &    \downarrow    \\
R/\mathrm J(R) & \overset{\overline t}\to & \mathcal O(R)/K         
\end{matrix}$$
 where $\overline t$ is deduced from $t$. Since $t$ is an epimorphism, so is $\overline t$, which is injective because $t^{-1}(K)=\mathrm J(R)$. Now, $R/\mathrm J(R)$ being zero-dimensional, $\overline t$ is surjective, and then $\overline t:R/\mathrm J(R)\to\mathcal O(R)/K$ is an isomorphism.
\end{proof}

 We are now looking at topological characterizations of $\mathrm J$-regular rings 
  and add a characterization of $\mathrm J$-regular rings  given in \cite{KZ}.

\begin{proposition}\label{5.25} Let $R$ be a ring. The following conditions are equivalent:
\begin{enumerate}
\item $R$ is $\mathrm J$-regular.
\item $\mathrm{Max}(R)$ is closed.
\item $\mathrm{Max}(R)$ is proconstructible.
\item $\mathrm{Max}(R)$ is compact for the flat topology.
\item \cite[Proposition 6.4, page 60]{KZ} For every $x\in R$, there exists $y\in R$ such that $xy\in\mathrm J(R)$ and $x+y\in \mathrm U(R)$.
\end{enumerate}
\end{proposition}
\begin{proof} (1) $\Leftrightarrow$ (4) by \cite[Theorem 4.5]{T}. 

(1) $\Rightarrow$ (2) Equality $\overline{\mathrm{Max}(R)}=\mathrm{V}(\mathrm J(R))$ always holds. Then (2) $\Leftrightarrow\ \overline{\mathrm{Max}(R)}={\mathrm{Max}(R)}$. But $\dim(R/\mathrm J(R))=0$ gives that $\mathrm{Max}(R)=\mathrm{V}(\mathrm J(R))$. Moreover, (1) $\Leftrightarrow\ \dim(R/\mathrm J(R))=0$ and $R/\mathrm J(R)$ is reduced, this last condition always holding since $\mathrm J(R)$ is semiprime. It follows that (1) $\Rightarrow\ \mathrm{Max}(R)=\mathrm{V}(\mathrm J(R))=\overline{\mathrm{Max}(R)}$, so that $\mathrm{Max}(R)$ is closed.

(2) $\Rightarrow$ (1) If $\mathrm{Max}(R)$ is closed, then $\mathrm{Max}(R)=\overline {\mathrm{Max}(R)}=\mathrm{V}(\mathrm J(R))$, so that $\dim(R/\mathrm J(R))=0$ which implies that $R$ is $\mathrm J$-regular.
 
 (2) $\Rightarrow$ (3) because a closed subset is proconstructible.
 
(3) $\Rightarrow$ (2) According to \cite[Corollaire 7.3.2, page 339]{EGA}, a proconstructible subset stable by specialization is closed.
 \end{proof}
 
\begin{corollary}\label{5.02} Let $R\subseteq S$ be an integral i-extension. Then   $R$ is $\mathrm J$-regular if and only if $S$ is $\mathrm J$-regular.
 \end{corollary}
\begin{proof} Since $R\subseteq S$ is an integral i-extension, the natural map $\mathrm {Spec}(S)\to\mathrm{Spec}(R)$ is an homeomorphism. Then it is enough to use the equivalence (1)$\Leftrightarrow$(2) of Proposition \ref{5.25}.
\end{proof}

 \begin{corollary}\label{5.26} A ring  $R$ whose spectrum is Noetherian for the flat topology, is  $\mathrm J$-regular.
\end{corollary}
\begin{proof} By \cite[Propositions 8 and 9, page 123]{ALCO}, $\mathrm{Max}(R)$ is Noetherian, and then compact for the flat topology. Then, use Proposition \ref{5.25}.
\end{proof}

\begin{corollary}\label{5.27} Let $f:R\to S$ be a ring morphism such that ${}^{a}f(\mathrm{Max}(S))$

\noindent$=\mathrm{Max}(R)$. If $S$ is $\mathrm J$-regular, so is  $R$.
\end{corollary}
\begin{proof} By Proposition \ref{5.25}, if $S$ is $\mathrm J$-regular, $\mathrm{Max}(S)$ is compact for the flat topology. Since ${}^{a}f$ is 
continuous
 for the flat topology, we get that ${}^{a}f(\mathrm{Max}(S))=\mathrm{Max}(R)$ is compact, so that $R$ is $\mathrm J$-regular.
\end{proof}

\begin{proposition}\label{5.142} A ring $R$ is $\mathrm J$-regular if and only if so is $R(X)$.
 \end{proposition}
\begin{proof} We know that $\mathrm{Max}(R(X))=\{MR(X)\mid M\in\mathrm {Max}(R)\}$, so that $\mathrm J(R(X))=\mathrm J(R)(X)$. It follows that $R(X)/\mathrm J (R(X))=R(X)/(\mathrm J(R)(X))$

\noindent$\cong (R/\mathrm J(R))(X)$. As $S$ is regular if and only if  so is $S(X)$, for a ring $S$, we get that $(R/\mathrm J(R))(X)$ is regular if and only if  so is $R/\mathrm J(R)$. Then, $R$ is $\mathrm J$-regular if and only if  so is $R(X)$.
\end{proof}

 According to \cite[Definitions 1.1, 1.2 and 1.5, Proposition 1.6]{Pic 0}, \cite{Pic 2} and \cite{S}, we set $p_r(X):=X^2-rX\in R[X]$, where $R\subseteq S$ is a ring extension.
$R\subseteq S$ is called {\it s-elementary} (resp.; {\it t-elementary}, {\it u-elementary}) if $S=R[b]$, where $p_0(b),bp_0(b)\in R$ (resp.; $p_r(b),bp_r(b)\in R$ for some $r\in R,\ p_1(b),bp_1(b)\in R$). In the following, the letter x denotes s, t or u. $R\subseteq S$ is called {\it cx-elementary} if $R\subseteq S$ is a tower of finitely many x-elementary extensions, {\it x-integral} if there exists a directed set $\{S_i\}_{i\in I}\subseteq[R,S]$ such that $R\subseteq S_i$ is cx-elementary and $S=\cup_{i\in I}S_i$. 
  An integral extension $R\subseteq S$ is called {\it infra-integral} \cite{Pic 2}, (resp. {\it subintegral} \cite{S}) if all its residual extensions are isomorphisms (resp$.$; and is an {\it i-extension}).    
 
An extension $R\subseteq S$ is called {\it s-closed} (or {\it seminormal}) (resp.; {\it t-closed}, {\it u-closed} (or {\it anodal})) if an element $b\in S$ is in $R$ whenever $p_0(b),bp_0(b)\in R$ (resp.; $p_r(b),bp_r(b)\in R$ for some $r\in R,\ p_1(b),bp_1(b)\in R$) \cite[Theorem 2.5]{S}. A ring $R$ is called {\it seminormal} by Swan if for $x,y\in R$, such that $x^2=y^3$, there is some $z\in R$ such that $x=z^3$ and $y=z^2$ \cite[Definition, page 210]{S}. We say that a ring $R$ is {\it t-closed} if for $x,y,r\in R$, such that $x^3+rxy-y^2=0$, there is some $z\in R$ such that $x=z^2-rz$ and $y=z^3-rz^2$ \cite[D\'efinition 1.1]{Pic 7}. A t-closed ring is seminormal.

A seminormal ring is reduced. We proved in \cite[Proposition 2.1]{Pic 7}
that a regular ring is t-closed, whence seminormal. 

Let x$\in\{$s,t,u$\}$. The {\it x-closure} ${}_S^xR$ of $R$ in $S$ is the smallest element $B\in[R,S]$ such that $B\subseteq S$ is x-closed and the greatest element $B'\in[R,S]$ such that $R\subseteq B'$ is x-integral. It follows that ${}_S^uR\subseteq{}_S^tR$. Note that the $s$-closure is actually the seminormalization ${}_S^+R$ of $R$ in $S$ and is the greatest subintegral extension of $R$ in $S$. Note also that the $t$-closure $ {}_S^tR$ is the greatest infra-integral extension of $R$ in $S$.

\begin{example}\label{5.03} (1) Let $R$ be a ring and $d$ a positive integer. Set $R_d:=\{\sum_{i\in I}\varepsilon_ia_i^d\mid\varepsilon\in\{1,-1\},\ a_i\in R,\ |I|<\infty\}$ which is a subring of $R$ such that $f:R_d\subseteq R$ is an integral extension. We claim that $f$ is an i-extension. Let $P,Q\in\mathrm{Spec}(R)$ be such that $P\cap R_d=Q\cap R_d$. Let $x\in Q$, so that $x^d\in Q\cap R_d=P\cap R_d\subseteq P$, which implies $x\in P$, and then $Q\subseteq P$. A similar proof shows that $P\subseteq Q$, so that $P= Q$. Then ${}^{a}f(\mathrm{Max}(R))=\mathrm{Max}(R_d)$ and $R$ is $\mathrm J$-regular if and only if $R_d$ is $\mathrm J$-regular by Corollary \ref{5.02}.

(2) Let $R\subseteq S$ be a u-closed FCP integral extension. Then $R\subseteq S$ is an i-extension by \cite[Proposition 5.2]{Pic 15} and $R$ is $\mathrm J$-regular if and only if $S$ is $\mathrm J$-regular by Corollary \ref{5.02}.
\end{example}

\begin{corollary} Let $R\subseteq S$ be a subintegral extension, such that $S$ is $\mathrm J$-regular. Then $R$ is $\mathrm J$-regular and there is a ring isomorphism $R/\mathrm J(R) \to S/\mathrm J(S)$. 
 In case $\mathrm J(R) =\mathrm J(S)$, then $R=S$. 
\end{corollary}
\begin{proof} $R/\mathrm J(R)$ is regular because reduced and zero-dimensional since so is $S/\mathrm J(S)$, and because the extension is integral. This also implies that $\mathrm J(R)=R\cap\mathrm J(S)$, so that the map $j: A:=R/\mathrm J(R)\to S/\mathrm J(S):=B$ exists, and its residual extensions are isomorphisms. Such residual extensions are of the form $A_P\to B_Q$ where $Q$ is a prime ideal of $B$ above $P$. It follows that $A\to B$ is a flat epimorphism since $A\to B$ is an $i$-extension \cite[Scholium A (1)]{Pic 5}. Finally, since this extension has the lying-over property for maximal ideals, $A\to B$ is a faithfully flat epimorphism, whence an isomorphism \cite[Lemme 1.2, page 109]{L}.
\end{proof}
 
We recall that a ring $R$ verifies the primitive condition if for any element $p(X)$ of the polynomial ring $R[X]$, whose content $c(p)$ is $R$ there is some $x\in R$ such that $f (x)\in\mathrm U(R)$ \cite{LOC1}. The Nagata ring $R(X)$ of a ring $R$ verifies the primitive condition \cite[p.457]{LOC1}.
 
Let $R\subseteq S$ be an extension. We denote by $\mathcal I(R,S)$ the abelian group of all $R$-submodules of $S$ that are invertible as in \cite[Definition 2.1]{LOC2}.

There is an exact sequence \cite[Theorem 2.4]{LOC2}:

\centerline{$1\to\mathrm U(R)\to\mathrm U(S)\to\mathcal I(R,S)\to\mathrm{Pic}(R)\to \mathrm{Pic}(S)$}

It follows that there is an injective map $\mathrm U(S)/\mathrm U(R)\to\mathcal I(R,S)$. 

\begin{proposition}\label{inv} Let $R\subseteq S$ be an extension. Then $R\subseteq S$ is an SL extension if and only if  $\mathcal I(R,S) = \{R\}$.         
 \end{proposition} 
\begin{proof} Use the injective map $\mathrm U(S)/\mathrm U(R)\to\mathcal I(R,S)$.
\end{proof}

\begin{proposition} The Picard group of a ring $R$ is $0$ if either $R$ is $\mathrm J$-regular or $R$ verifies the primitive condition. Moreover, 
 over such ring, a finitely generated projective module of finite (local) constant rank is free.
\end{proposition}
\begin{proof} It is enough to combine \cite[Propositions p. 455 and p. 456]{LOC1} and \cite[Theorem and Corollary p. 457]{LOC1}.
\end{proof}

\begin{corollary} \cite[Remark 2.5]{LOC2} Let $R\subseteq S$ be an extension, where $\mathrm{Pic}(R)=0$ (for example, if either $R$ is a Nagata ring or $R$ is $\mathrm J$-regular), then $\mathrm U(S)/\mathrm U(R)$  is isomorphic to $\mathcal I(R,S)$.
\end{corollary}

\begin{proposition} Let $R\subseteq S$ be an extension and $M$ an $R$-submodule of $ S$. Then $M$ belongs to $\mathcal I(R,S)$ if and only if $M$ is an $R$-module of finite type, $MS=S$ and $M$ is projective of rank one. When $R$ is either a Nagata ring or $R$ is $\mathrm J$-regular, the last condition can be replaced with $M$ is free of dimension one and then the elements of $\mathcal I(R,S)$ are of the form $Ru$ where $u$ is a unit of $S$.
\end{proposition}

\begin{proof} This a consequence of \cite[Lemma 2.2 and Lemma 2.3]{LOC2}.
\end{proof}

\begin{remark}\label{5.10} We now consider an SL extension $R \subseteq S$. 

(1) Since $x\in R$ belongs to $\mathrm J(R)$ if and only if $1-ax$ is a unit for any $a\in R $, let $x\in\mathrm J(S)$. In particular, $1-x\in\mathrm U(S)=\mathrm U(R)$, which implies $x\in R$. It follows easily that $\mathrm J(S)\subseteq\mathrm J(R)$. Therefore $\mathrm J(S)$ is an ideal shared by $R$ and $S$. We infer from this fact that if $P$ is a prime ideal of $R$, that does not contains $\mathrm J(S)$, we have $R_P= S_P$. 

(2) If $2\in\mathrm U(S)$, we claim that $R$ and $S$ have the same idempotents: Let $ e$ be an idempotent of $S$. Then, $(1-2e)^2=1-4e+4e^2=1$, so that $1-2e\in\mathrm U (S)=\mathrm U(R)$, which implies $2e\in R$. Since $2\in\mathrm U(S)=\mathrm U(R)$, we get that $e\in R$.

(3) If the class of an element $s\in S$ is a unit of $S/\mathrm J(S)$, there is some $t\in S$, such that $st-1\in\mathrm J(S)$. It follows that $s$ cannot belong to any maximal ideal of $S$ and is therefore a unit of $S$ and then of $R$. We deduce from these facts that $\mathrm U(R/\mathrm J(S))=\mathrm U(S/\mathrm J(S))$ and $R/\mathrm J(S)\subseteq S/\mathrm J(S)$ is SL.
\end{remark} 

\begin{proposition}\label{5.71} A ring extension $R\subseteq S$ is SL if and only if $\mathrm{Nil}(R)=\mathrm{Nil}(S)$ and $R/\mathrm{Nil}(R)\subseteq S/\mathrm{Nil}(S)$ is SL.
 \end{proposition}
\begin{proof} Assume that $R\subseteq S$ is SL. Then $\mathrm{Nil}(R)=\mathrm{Nil}(S)$ by Proposition \ref{3.28}. Set $R':=R/\mathrm{Nil}(R)$ and $S':=S/\mathrm{Nil}(S)$. We claim that $R'\subseteq S'$ is SL. Let $\overline x\in\mathrm U(S')$, where $\overline x$ is the class of $x\in S$. There exists $y\in S$ such that $\overline x\overline y=\overline 1\ (*)$ in $S'$, so that $1-xy\in\mathrm{Nil}(S)=\mathrm{Nil}(R)\subseteq \mathrm J(S)$. Then $1-xy\in\mathrm J(S)$ implies $x,y\in\mathrm U(S)=\mathrm U(R)\subseteq R$. As $1-xy\in\mathrm{Nil}(R)$ with $x,y\in R$, we deduce from $(*)$ that $\overline x,\overline y\in\mathrm U(R')$, giving that $\mathrm U(S')=\mathrm U(R')$ and $R/\mathrm{Nil}(R)\subseteq S/\mathrm{Nil}(S)$ is SL.

Conversely, assume that $\mathrm{Nil}(R)=\mathrm{Nil}(S)$ and $R/\mathrm{Nil}(R)\subseteq S/\mathrm{Nil}(S)$ is SL. By Corollary \ref{5.1},  $R\subseteq S$ is  SL.
\end{proof}
 
\begin{theorem}\label{5.2} Let $R\subseteq S$ be an SL extension. Then $S$ is $\mathrm J$-regular if and only if $R$ is $\mathrm J$-regular and $R\subseteq S$ is integral seminormal. If these conditions hold, then $\mathrm J(R)=\mathrm J(S)$.
\end{theorem}
\begin{proof} (1) According to Remark \ref{5.10}(3), $R/\mathrm J(S)\subseteq S/\mathrm J(S)$ is SL. 
Assume that $S$ is $\mathrm J$-regular.
We want to show that $R$ is $\mathrm J$-regular with $R\subseteq S$ integral seminormal. Since $S/\mathrm J(S)$ is regular, it is enough to suppose that the extension $R\subseteq S$ is SL with $S$ regular and to show that $R$ is regular.

Let $r$ be in $R$. As an element of $S$, there is some $r'\in S$ such that $r^2r'=r$; so that $e=rr'$ is idempotent and $1-e+r\in\mathrm U(S)=\mathrm U(R)\subseteq R$, with $(1-e+r)^{-1}=1-e+r'$ (see (1) at the beginning of Section 5). We have clearly $e\in R$ and, since $1-e+r'\in\mathrm U(S)=\mathrm U(R)\subseteq R$, then $r'\in R$. To conclude, $R$ is regular. 

Let $x\in S$. Then, $x=ue,\ u\in\mathrm U(S)=\mathrm U(R)$ and $e$ an idempotent of $S$, that is $e^2-e=0$, so that $e$ is integral over $R$, and so is $x$. Then, $R\subseteq S$ is integral. According to the absolute flatness of $R$ and $S$, we deduce that these rings are seminormal. It follows that $R\subseteq S$ is seminormal by \cite[Corollary 3.4]{S}.

Coming back to the first case, it follows that if $S$ is $\mathrm J$-regular, so is $R$. Since $R/\mathrm J(S)\subseteq S/\mathrm J(S)$ is integral seminormal, so is $R\subseteq S$.

(2) We intend to show that $\mathrm J(R)=\mathrm J(S)$. Returning to the Jacobson ideals, we see that $\mathrm J(S)$ is an intersection of maximal ideals of $R$ because $ R/\mathrm J(S)$ is regular and it follows that $\mathrm J(R)\subseteq\mathrm J(S)$, the lacking inclusion. By the way, we have shown that a maximal ideal of $S$ contracts to a maximal ideal of $R$. Actually the natural map $\mathrm{Max}(S)\to\mathrm{Max}(R)$ is surjective because a minimal prime ideal can be lifted up in the extension $R/\mathrm J(R)\subseteq   S/\mathrm J(S)$. 

(3) Conversely, assume that $R$ is $\mathrm J$-regular and $R\subseteq S$ is integral seminormal. By Remark \ref{5.10}(1), we know that $\mathrm J(S)\subseteq\mathrm J (R)$ and $\mathrm J(S)$ is an ideal of $R$, an intersection of the maximal ideals of $S$. Then, $\mathrm J(S)=\mathrm J(S)\cap R$ is also an intersection of maximal ideals of $ R$ since $R\subseteq S$ is integral. Hence, $\mathrm J(R)\subseteq\mathrm J(S)$ giving $\mathrm J(R)=\mathrm J(S)$. Setting $R':=R/\mathrm J(R)$ and $S':=S/\mathrm J(S)$, we get that $R'\subseteq S'$ is integral seminormal and SL by Remark  \ref{5.10}(3). Since $R'$ is regular, it is reduced and we have $\dim(R')=0=\dim(S')$ because $R'\subseteq S'$ is integral. We claim that $S'$ is reduced. Assume there exists some $x\in S',\ x\neq 0$ which is nilpotent, and let $n$ be the least integer $>1$ such that $x^n=0$ with $y:=x^{n-1}\neq 0$. We get that $y^2=x^{n+(n-2)}=0=y^3\in R'\ (*)$. Since $R'\subseteq S'$ is  seminormal, it follows that $y\in R'$. But $(*)$ implies that $y= 0$ since $R'$ is reduced, a contradiction. We have that $S'$ is reduced, and then regular, so that $S$ is $\mathrm J$-regular.
 \end{proof} 

Lequain and Doering have considered in \cite{DL} SL extensions $R \subseteq S$ such that $S$ is semilocal. Then Theorem \ref{5.2}  implies  one of their results,  because $S$ is $\mathrm J$-regular by the Chinese remainder theorem.
We will improve \cite[Theorem 1]{DL} by using our methods.

\begin{proposition}\label{5.3} An SL extension $R\subseteq S$, such that $S$ is regular,  is u-integral, infra-integral,  seminormal, quadratic and $R$ is regular.
\end{proposition}
\begin{proof} Since $S$ is regular, any $x\in S$ can be written $x=eu$, where $e$ is an idempotent and $u\in\mathrm U(S)$. Set $T:={}_S^uR$. Then, $T\subseteq S$ is u-closed. Since $e^2-e=e^3-e^2=0\in T$, we get that $e\in T$. Moreover, $\mathrm U(R)\subseteq\mathrm U(T)\subseteq\mathrm U(S)=\mathrm U(R)$ gives $\mathrm U(T)=\mathrm U(S)$, so that $u\in T$, and then $x\in T$, whence $T=S$ and $R\subseteq S$ is u-integral. Since $S={}_S^uR\subseteq{}_S^tR\subseteq S$, we get $S={}_S^tR$ and $R\subseteq S$ is infra-integral. Since $S$ is regular, so is $S/\mathrm J(S)$ and $R\subseteq S$ is seminormal by  Theorem \ref{5.2}. 

Let $x\in S$ and $y\in S$ be such that $x^2y=x$, since $S$ is regular. Keeping the previous notation and properties of the beginning of the Section, $e:=xy$ is the idempotent such that $x=eu$, with $u\in\mathrm U(S)=\mathrm U(R)$. It follows that $e= xu^{-1}$, so that $x^2u^{-1}=xe=x^2y=x\ (*)$. Then, $x^2=xu$ gives that any $x\in S$ is quadratic and $R\subseteq S$ is a quadratic extension. Now, if $x\in R,\ (*)$ is still valid with $u\in\mathrm U(R)$ as we have just seen. Then, $x^2u^{-1}=x$ shows that $R$ is regular.  
\end{proof}

\begin{corollary}\label{5.30} Let $R\subseteq S$ be an SL extension such that $S$ is $\mathrm J$-regular. Then $R\subseteq S$ is u-integral, infra-integral and quadratic.
\end{corollary}
\begin{proof} Since $R\subseteq S$ is SL, so is $R/\mathrm J(S)\subseteq S/\mathrm J(S)$ by Remark \ref{5.10} (3). Then, we may apply the results of Proposition \ref{5.3} to the extension $R/\mathrm J(R)\subseteq S/\mathrm J(S)$ because $\mathrm J(R)=\mathrm J(S)$ by Theorem \ref{5.2}, which gives that $R/\mathrm J(R)\subseteq S/\mathrm J(S)$ is u-integral, infra-integral and quadratic. It follows that $R\subseteq S$ also verifies these properties. 
\end{proof}
 
Recall that a ring extension $R\subseteq S$ is called a $\Delta_0$-extension if $T\in[R, S] $ for each $R$-submodule $T$ of $S$ containing $R$ \cite{HP}. 

\begin{corollary}\label{5.31} Let $R\subseteq S$ be an SL extension such that $S$ is regular. The following properties hold:
\begin{enumerate}
\item For any $Q\in\mathrm{Spec}(S)$ and $P:=Q\cap R$, we have $R_P\cong S_Q$. If, moreover, $R\subseteq S$ is an i-extension, then $R=S$.
\item If $R\subseteq S$ is u-closed, then $R=S$.
\item If $R\subseteq S$ is simple, then  $R\subseteq S$ is a $\Delta_0$-extension.
\item If $2\in \mathrm U(S)$, then $R=S$.
\end{enumerate}
\end{corollary}
\begin{proof} (1) By  Proposition \ref{5.3}, $R\subseteq S$ is infra-integral and $ R$ is regular. It follows that for any $Q\in\mathrm{Spec}(S)$ and $P:=Q\cap R$, we have an isomorphism $R/P\cong S/Q$. But, $R$ and $S$ being regular, they are zero-dimensional, so that $R/P\cong R_P$ and $S/Q\cong S_Q$ giving $R_P\cong S_Q$.  

Assume, moreover, that $R\subseteq S$ is an i-extension, then $R\subseteq S$ is a flat epimorphism \cite[Scholium A (1)]{Pic 5} or \cite[Ch. IV]{L}. Then, $R=S$ by \cite[Scholium A (2)]{Pic 5}.

(2) According to Proposition \ref{5.3}, $R\subseteq S$ is u-integral. Assume, moreover, that $R\subseteq S$ is u-closed, then $R={}_S^uR=S$.

(3) According to Proposition \ref{5.3}, $R\subseteq S$ is quadratic. Assume, moreover, that $R\subseteq S$ is a simple extension. Set $S:=R[y]$, where $y$ is a quadratic element over $R$. Then, $S=R+Ry$ and it follows from \cite[Proposition 4.12]{Gil} that $R\subset S$ is a $\Delta_0$-extension.

(4) Since $S$ is regular, according to \cite[Theorem 2.10]{AB}, any $x\in S$ can be written $x=u+v$, where $u,v\in\mathrm U(S)=\mathrm U(R)$ which implies $x\in R$ and $R=S$.
\end{proof}

\begin{proposition}\label{6.5} Let  $M$ be an $R$-module. The following conditions hold:
\begin{enumerate}
\item $\mathrm J(R (+)  M)=\mathrm J(R) (+)  M$.
\item $R (+)  M$ is $\mathrm J$-regular  if and only if $R$ is $\mathrm J$-regular. 
\end{enumerate}
\end{proposition} 
\begin{proof} (1) comes from \cite[Theorem 3.2]{AW} because any maximal ideal of $R (+)  M$ is of the form $P (+)  M$ where $P\in \mathrm{Max}(R)$.

(2) $R(+)M$ is $\mathrm J$-regular if and only if $(R(+)M)/(\mathrm J(R)(+)M)$ is regular. But $(R(+)M)/(\mathrm J(R)(+)M)\cong(R/\mathrm J(R))(+)(M/M)\cong R/\mathrm J(R)$. By \cite[Theorem 3.1]{AW}, we get the result. 
\end{proof}

A ring $R$ is called a {\it $\mathrm{Max}$-ring} if any nonzero $R$-module has a maximal submodule. An ideal $I$ of $R$ is {\it $\mathrm T$-nilpotent} if for each sequence $\{r_i\}_{i=1}^{\infty}\subseteq I$, there is some positive integer $k$ with $r_1\cdots r_k=0$. 

\begin{proposition} \label{5.127} Let $R\subseteq S$ be an SL extension where $S$ is a $\mathrm{Max}$-ring. Then $R$ is also a $\mathrm{Max}$-ring and $R$ and $S$ are $\mathrm J$-regular. 
\end{proposition}
\begin{proof} $S$ is regular and $\mathrm J(S)$ is $\mathrm T$-nilpotent 
since $S$ is a $\mathrm{Max}$-ring by \cite[Theorem, p.1135]{Ham}. Then, $S$ is also $\mathrm J$-regular, and so is $R$, with $\mathrm J(R) =\mathrm J(S)$ by Theorem \ref{5.2}. It follows that $\mathrm J(R)$ is $\mathrm T$-nilpotent. The same reference gives that $R$ is a $ \mathrm {Max}$-ring since $R$ is regular by Proposition \ref{5.3}.
\end{proof} 

\section{Cohn's rings}

In \cite[Theorem 1]{ZCOHN}, P. M. Cohn shows that for each ring $R$ there is an SL extension $R\subseteq R'$ such that  $\mathrm Z (R')=R'\setminus\mathrm U(R')$. We generalize  some results.  
 
\begin{lemma} \label{1Cohn} Let $I$ be a semiprime ideal of a ring $R$ and $f(X),g(X)\in R [X],\ a,b\in\mathrm U(R)$ such that $af(X)+bg(X)+Xf(X)g(X)\in I[X]\ (*)$. Then $f(X),g(X)\in I[X]$.
 \end{lemma}
 \begin{proof} We first prove the Lemma for a prime ideal $P$. We begin to remark that if $ f(X)=f_1(X)+f_2(X)$, with $f_2(X)\in P[X]$, then $(*)$ is equivalent to $af_1(X)+bg(X)+Xf_1 (X)g(X)\in P[X]$. The same property holds for $g(X)$. Then, the conclusion of the Lemma will result from the fact that $f_1(X),g(X)\in P[X]$. Moreover, once we have proved that $f (X)\in P[X]$, condition $(*)$ shows that $bg(X)\in P[X]$, which gives $g(X)\in P[X]$ because $b\in\mathrm U(R)$. So, assume that $f(X)$ and $g(X)\not\in P[X]$. Hence, we can set $f(X):=\sum_{i=0}^n\alpha_iX^{i}$ and $g(X):=\sum_{j=0}^p\beta_jX^ j$ with $\alpha_n$ and $\beta_p\not\in P$, according to the precedent discussion. The only term of highest degree in $(*)$ is $\alpha_n\beta_pX^{n+p+1}$ which is in $P[X]$, so that $\alpha_ n\beta_p\in P$, a contradiction with $\alpha_n$ and $\beta_p\not\in P$. To conclude, $f(X),g(X)\in P [X]$.
 
Now, assume that $af(X)+bg(X)+Xf(X)g(X)\in I[X]\ (*)$ for a semiprime ideal $I$. Then, $I=\cap_{\lambda\in\Lambda}P_{\lambda}$, for a family of prime ideals $\{P_{\lambda}\}_ {\lambda\in\Lambda}$. Condition $(*)$ shows that $af(X)+bg(X)+Xf(X)g(X)\in P_{\lambda}[X]$ for each $\lambda\in\Lambda$. By the first part of the proof, it follows that $f(X),g(X)\in P_{\lambda}[X]$ for each $\lambda\in\Lambda$, which implies that $f(X),g(X)\in\cap_ {\lambda\in\Lambda} P_{\lambda} [X]=I[X]$.
\end{proof}

\begin{proposition} \label{2Cohn} For any ring $R$ there exists an SL extension $R\subseteq R'$ such that any nonunit of $R$ is a zerodivisor in $R'$. The extension $R\subseteq R'$ is pure and t-closed 
and  $(R:R')=\mathrm J(R).$
 If $R$ is reduced, so is $R'$. 
 \end{proposition}
\begin{proof}  For each $M\in\mathrm{Max}(R)$, let $X_M$ be an indeterminate attached to $M$. Set $\mathcal X:=\{X_M\mid M\in\mathrm{Max}(R)\}$. We consider $R[\mathcal X]$ as the ring of polynomials in the indeterminates $X_M\in\mathcal X$. At last, define $I:=\sum_{M\in\mathrm{Max}(R)}X_MMR[\mathcal X]$, which is an ideal of $R[\mathcal X]$ and $R':=R[\mathcal X]/I$. For each $M\in\mathrm{Max}(R)$, we define $x_M$ as the class of $X_M$ in $R'$. We prove the Proposition in several steps.
 
 (1) Let $M,N\in\mathrm{Max}(R),\ M\neq N$. We claim that $x_Mx_N=0$ in $R'$. 
 
Since $M\neq N$, they are comaximal, and there exist $m\in M,\ n\in N$ such that $m+n=1 $. It follows that $X_MX_N=(m+n)X_MX_N=mX_MX_N+nX_NX_M\in I$, giving $x_Mx_N= 0$. Then, the class of a polynomial of $R[\mathcal X]$ is written as the class of a finite sum of polynomials in one indeterminate, that is in finitely many $X_M$.
 
(2) $R$ is a subring of $R'$ because the composite ring morphism $R\subseteq R [\mathcal X]\to R[\mathcal X]/I$ is injective. Indeed, $R\cap I=0$ since the constant term of any polynomial in $I$ is 0. Then, we can identify the class of the constant term of any polynomial of $R[\mathcal X]$ to this constant term.
 
(3) Any $y\in R'$ can be written in an unique way 
$$y=a+\sum_{M\in\Lambda}x_Mf_M(x_M)$$
where $a\in R,\ f_M$ is a polynomial of $R[X]$ and the set $\Lambda$ is a finite subset of $\mathrm{Max}(R)$. 
 
The existence of this writing results from (1) and (2). Assume that there exist $b\in R,\ g_M \in R[X]$ such that $y=b+\sum_{M\in\Lambda}x_Mg_M(x_M)$. There is no harm to choose the same $\Lambda$ for the two writings. Then, $a+\sum_{M\in\Lambda}x_Mf_M(x_M)=b+\sum_{M\in\Lambda}x_Mg_M(x_M)$ in $R'$ gives that $(a-b)+\sum_{M\in\Lambda}X_M[f_ M(X_M)-g_M(X_M)]\in I=\sum_{M\in\mathrm{Max}(R)}X_MMR[\mathcal X]$. First, we get $a=b$ by the substitution $X_M\mapsto 0$ for each $M$. Now, we get that $X_M[f_M(X_ M)-g_M(X_M)]\in X_MMR[X_M]$ for each $M\in\Lambda$, so that $x_Mf_M(x_M) =x_Mg_M(x_M)$ for each $M\in\Lambda$ showing the uniqueness of the writing.

By the way we proved that $R'=R\bigoplus_{M\in\Lambda}x_MR[x_M]$.
 
 (4) We prove that $(R:R')=\mathrm J(R)$.
 
Let $t\in R$. Then $t\in(R:R')$ if and only if $tx_M\in R$ for any $M\in\mathrm{Max}(R)$ if and only if $tx_M=0$ by (3), which is equivalent to $t X_M\in I$ for any $M\in\mathrm{Max}(R)$, that is $t\in\cap_{M\in\mathrm {Max}(R)}M=\mathrm J(R)$. It follows that $(R:R')=\mathrm J(R).$

 (5) $R\subseteq R'$ is SL. 
 
Let $y\in\mathrm U(R')$. There exists $z\in\mathrm U(R')$ such that $yz= 1\ (*)$. Write $y=a+\sum_{M\in\Lambda}x_Mf_M(x_M)$ and $z=b+\sum_ {M\in\Lambda}x_Mg_M(x_M)$, where $a,b\in R$ and $f_ M,\ g_M$ are polynomials of $R[X]$. We can choose the same $\Lambda$ for $y$ and $z$. Then $(*)$ becomes $1=(a+\sum_{M\in\Lambda}x_Mf_M(x_ M))(b+\sum_{M\in\Lambda}x_Mg_M(x_M))=ab+\sum_{M\in\Lambda}x_M[ag_M(x_M)+bf_M(x _M)+x_Mf_M(x_M)$

\noindent$g_M(x_M)]$ 
 from which it follows that $ab=1\ (**)$ and $ag_M(X_M)+bf_M(X_M)+ X_Mf_M(X_M)g_M(X_M)\in MR[X_M]\ (***)$. By $(**)$, we get that $a,b\in\mathrm U(R)$. Then, we may apply Lemma \ref{1Cohn} to $(***)$, and we get that $f_M(X_M)$ and $g_M(X_M)$ are in $MR[X_M]$, so that $ x_Mf_M(x_M)=x_Mg_M (x_M)=0$ and $y=a\in\mathrm U(R)$. Then, $\mathrm U(R)=\mathrm U(R')$ and $R\subseteq R'$ is SL.
 
(6) Any nonunit of $R$ is a zerodivisor in $R'$. 

Obvious: If $y\in R$ is not a unit in $R$, there exists some $M\in\mathrm{Max}(R)$ such that $y\in M$. Then $X_My\in X_MMR[X_M]$ giving $x_My=0$ in $R'$ with $x_M\neq 0$, so that $y$ is a zerodivisor in $R'$.

(7) $R\to R'$ is pure. 

Consider the composite ring morphism $R\subseteq R[\mathcal X]\to R[\mathcal X]/I$ and let $\varphi:R[\mathcal X]\to R$ be the ring morphism defined by the substitution $X _M\mapsto 0$ for each $M$. Since  $I$ is contained in the kernel of $\varphi$, there is a  ring morphism $R[\mathcal X]/I\to R$. It follows that $R[\mathcal X]/I$ is a retract of $R$ and the extension $R\to R[\mathcal X]/I$ is pure. 

(8)  $R\subseteq R'$ is t-closed. 

Let $y\in R'$ be such that $y^2-ry,y^3-ry^2 \in R$ for some $r\in R$. As in the beginning of the proof, set $y=a+\sum_{M\in\Lambda}x_Mf_M(x_M)$ where $a\in R,\ f_M(X_M):=\sum_{i=0}^n\alpha_iX_M^{i}$ is a polynomial of $R[X_M],\ n:=\deg(f_M)$ and the set $\Lambda$ is a finite subset of $\mathrm{Max}(R)$. Then, $y^2-ry=(a^2-ra)+\sum_{M\in\Lambda}[(2a-r) x_Mf_M(x_M)+x_M^2f_M(x_M)^2]$. Assume that $y\not\in R$ so that there exists some $M\in\mathrm{Max}(R)$ such that $x_Mf_M(x_M)\neq 0$. As in the proof of Lemma \ref{1Cohn}, we may assume that $\alpha_n\not\in M$. But $y^2-ry\in R$ implies that $\alpha_n^2X_M^{2n+2}\in X_M MR[X_M]$, whence $\alpha_n^2\in M$ and $\alpha_n\in M$, a contradiction. Then, any $y\in R'$ such that $y^2-ry,y^3-ry^ 2\in R$ is in $R$ and $R\subseteq R'$ is t-closed, and, in particular, seminormal. This implies that $R'$ is reduced when $R$ is reduced. 
\end{proof}

\begin{corollary}\label{0Cohn} Let $R\subseteq S$ be a ring extension such that the spectral map $\mathrm{Max}(S)\to\mathrm{Max}(R)$ is bijective and let $R'$ (resp. $S'$) be the ring associated to $R$ (resp. $S$) gotten in Proposition \ref{2Cohn}. Then, $R\subseteq S$ is SL if and only if $R'\subseteq S'$ is SL. 
 \end{corollary}
\begin{proof} For each $M\in\mathrm{Max}(R)$, there exists a unique $N\in\mathrm {Max}(S)$ lying above $M$, and for each $N\in\mathrm{Max}(S)$, then, $N\cap R\in\mathrm{Max}(R)$. We use the notation of Proposition \ref{2Cohn} and consider for each $M\in\mathrm{Max}(R)$, an indeterminate $X_M$ attached to $M$. Then, we may say that $X_M$ is also attached to any $N\in\mathrm{Max}(S)$ such that $N\cap R=M\in\mathrm{Max}(R)$. Then, the $X_M$'s are attached in a unique way to all maximal ideals of $S$. Set $\mathcal X:=\{X_M\mid M\in\mathrm{Max}(R)\}$. We consider $R [\mathcal X]$ as the ring of polynomials in the indeterminates $X_M\in\mathcal X$ over $ R$ and $S[\mathcal X]$ as the ring of polynomials in the indeterminates $X_M\in\mathcal X$ over $S$. Setting $I:=\sum_{M\in\mathrm{Max}(R)}X_MMR[\mathcal X]$, we define $J:=\sum_{M\in\mathrm{Max}(R)}[X_MNS[\mathcal X]\mid M=N\cap R]$ which is an ideal of $S[\mathcal X]$. Set $R':=R[\mathcal X]/I$ and $S':=S[\mathcal X]/J$. We get the following commutative diagram:
$$\begin{matrix}
                R            & \to  &           S               \\
        \downarrow    & {}   &    \downarrow       \\
      R[\mathcal X]   & \to &   S[\mathcal X]      \\
       \downarrow     & {}  &    \downarrow        \\
R'=R[\mathcal X]/I & \to & S'=S[\mathcal X]/J         
         \end{matrix}$$
where all the maps are extensions since $I=J\cap R[\mathcal X]$. By Proposition \ref{2Cohn}, $R\subseteq R'$ and $S\subseteq S'$ are both SL extensions. Now, applying Proposition \ref{3.1}, if $R\subseteq S$ is SL, so is $R\subseteq S'$ and then also $R'\subseteq S'$. By the same Proposition, if $R'\subseteq S'$ is SL, so is $R\subseteq S'$ and then also $R\subseteq S$ because $S\subseteq S'$ is injective.
         \end{proof}

 \begin{theorem} \label{3Cohn} Let $R$ be a ring. There exists an SL extension $R\subseteq S$ such that any nonunit of $S$ is a zerodivisor in $S$. The extension $R\subseteq S$ is pure and t-closed. If, moreover, $R$ is reduced, so is $S$.
 \end{theorem}
\begin{proof} According to Proposition \ref{2Cohn}, there exists an SL extension $R\subseteq R'$ such that any nonunit of $R$ is a zerodivisor in $R'$. Set $R_0:=R$ and $ R_1=R'$. We build by induction, with the same Proposition, a chain $\{R_i\}_{i\in I}$ defined by $R_{i+1}:=R_i'$ verifying $\mathrm U(R_{i+1})=\mathrm U(R_i)$. The induction shows that $\mathrm U(R_i)=\mathrm U(R)$. Now any nonunit of $R_i$ is a zerodivisor in $R_{i+1}$. Setting $S:=\cup_{i\in I}R_i$, we get an extension $R\subseteq S$. 

Let $a\in\mathrm U(S)$ and set $b:=a^{-1}\in\mathrm U(S)$, so that $ab=1\ (*)$ in $S$. There exists some $i\in I$ such that $a,b\in R_i$ (we can take the same $i$ for $a$ and $b$). Then $(*)$ still holds in $R_i$, so that $a\in\mathrm U(R_i)=\mathrm U(R)$, giving $\mathrm U(S)=\mathrm U(R)$ and $R\subseteq S$ is SL.

Let $x$ be a nonunit of $S$. There exists some $i\in I$ such that $x\in R_i$. Obviously, $ x$ is a nonunit of $R_i$, and then is a zerodivisor in $R_{i+1}$, and also a zerodivisor in $S$. So, any nonunit of $S$ is a zerodivisor in $S$. 

Since $R\subseteq R'$ is pure and t-closed by Proposition \ref{2Cohn}, so is $R_i\subseteq R_{i+1}$ for any $i$. Now, $R\subseteq S$ is pure (resp. t-closed) since $R\subseteq S$ is the union of pure (resp. t-closed) morphisms $R_i\subseteq R_{i+1}$.

If, moreover, $R$ is reduced, so is $R'$ by Proposition \ref{2Cohn}, and so are any $R_i$, and to end, so is $S$.
 \end{proof}

 \begin{remark}\label{4Cohn} The ring gotten in Theorem \ref{3Cohn} is not the only one verifying conditions of Theorem \ref{3Cohn}. For example any regular ring $R$ satisfies these conditions as the ring $S$ built from $R$ in Theorem \ref{3Cohn}.
\end{remark}
 
 We now  build SL extensions from a new type of rings whose construction is close to Cohn's rings. Many proofs are similar. If $I$ is an ideal of a ring $R$, we set  
 $R/\!\!/I:=R[X]/XI[X]$.  
 Since $XI[X] \cap R = 0$, we can consider that we have an extension $R \subseteq R/\!\!/I$, with $R \subset R/\!\!/I$ when $I\neq R$ because $X\in XR[X]\setminus XI[X]$ shows that $X-a\not\in XI[X]$ for any $a\in R$.  
 
\begin{lemma} \label{Cohn} If $P$ is a prime ideal of a ring $R$, the extension $R\subseteq R/\!\!/P$ is SL, pure, t-closed and $R/\!\!/P$ is reduced if so is $R$. Moreover, $P =(R:R/\!\!/P),\ R/\!\!/P\cong R\bigoplus (X(R/P)[X])$ and if $P\not\subseteq Q$, then, $\kappa(Q)\otimes_R (R[X]/XP[X])\cong\kappa(Q)$ and if $P\subseteq Q$, then, $\kappa(Q)\otimes_R(R[X]/XP[X])\cong\kappa(Q)[X]$ for any $Q\in\mathrm{Spec}(R)$.
 
Any  $a\in R$ is regular in $R/\!\!/P$ if and only if $a$ is regular in $R$ and $a\not\in P$. 
\end{lemma}
\begin{proof} Set $P':=XP[X]$. Any $y\in P'$ can be written $y=\sum_{i=1}^np_iX^{i},\ p_i$
 
\noindent $\in P\ (*)$. Then, looking at the terms of degrees 0 and 1 in $(*)$, we get $P'\cap R=0$ and $X\not\in P'$.   

An element of $R/\!\!/P=R[X]/XP[X]$ is of the form $y=a+xf(x)$, where $x$ is the class of $X$ in $R/\!\!/P$ and $a\in R$. As $R\cap XR[X]=0$, we also have $R\cap xR[x]=0$, so that $ R/\!\!/P=R\bigoplus(xR[x])$, with $x P=0$, a direct sum of $R$-modules. Obviously, $XR[X]/XP[X]\cong X(R/P)[X]$ as $R$-modules.

Now, $y\in\mathrm U(R/\!\!/P)$ if and only if there exists $z\in R/\!\!/P$ such that $yz=1\ (**)$. Set $y=a+xf(x)$ and $z=b+xg(x)$ with $f(X),g(X)\in R[X]$. Then $(**)\Leftrightarrow(a+x f(x))(b+xg(x))=1\Leftrightarrow ab+x[ag(x)+bf(x)+xf(x)g(x)]=1\Leftrightarrow ab=1\ (i)$ and $ag(X)+bf(X)+Xf(X)g(X)\in P[X]\ (ii)$. Since $(i)\Leftrightarrow a,b\in\mathrm U(R)$, then $(i)$ and $(ii)$ implies $f(X),g(X)$

\noindent$\in P[X]$, according to Lemma \ref{1Cohn}, so that $y=a,\ z=b\in\mathrm U(R)$ which give $\mathrm U(R/\!\!/P)=\mathrm U(R)$ and $R\subseteq R/\!\!/P$ is  SL. 

We now prove the second statement. Let $\varphi:R[X]\to R$ be the ring morphism defined by $X\mapsto 0$. Since $P'\subseteq\ker(\varphi)$, there is a ring morphism $R/\!\!/P\to R$. It follows that $R/\!\!/P$ is a retract of $R$ and the extension $R\to R/\!\!/P$ is pure. 

We claim that $R\subseteq R/\!\!/P$ is t-closed. Let $y\in R/\!\!/P$ be such that $y^2-ry,y^3-ry^2\in R$ for some $r\in R$. Assume that $y\not\in R$. According to the beginning of the proof, we can write $y=a+xf(x)$ where $a\in R,\ f(X):=\sum_{i=0}^n\alpha_iX^{i}$ is a polynomial of $R[X],\ n:=\deg(f)$ and either $\alpha_i=0$ or $\alpha_i\not\in P$, with in particular, $\alpha_n\not\in P$. Then, $y^2-ry=(a^2-ra)+(2a-r)xf(x)+x^2f(x)^2$. But $y^2-ry\in R$ implies that $\alpha_n^2X^{2n+2}\in XPR[X]$, so that $\alpha_n^2\in P$ and $\alpha_n\in P$, a contradiction. Then, any $y\in R/\!\!/P$ such that $y^2-ry,y^3-ry^2\in R$ is in $R$ and $ R\subseteq R/\!\!/P$ is t-closed, and, in particular, is seminormal, whence $R/\!\!/P$ is reduced when $R$ is reduced. 

Now, let $a\in R$. Then, $a\in(R:R/\!\!/P)\Leftrightarrow aX^k\in R+XP[X]$ for each integer $k\geq 1\Leftrightarrow a\in P$, so that $P=(R:R/\!\!/P)$ 
 since $P(R:R/\!\!/P)\subseteq R$. 

Set $S:=R[X]$, so that $R/\!\!/P=S/P'$. Let $Q\in\mathrm{Spec}(R)$. Then, $\kappa(Q)\otimes_RS/P'=(S_Q/P'_Q)/Q(S_Q/P'_Q)\cong(S_Q/P'_ Q)/((QS_Q+P'_Q)/P'_Q)\cong S_Q/(QS_Q+P'_Q)\cong R_Q[X]/(QR_Q[X]+XPR_Q[X])$.

If $P\not\subseteq Q$, then, $XPR_Q[X]=XR_Q[X]$ and $\kappa(Q)\otimes_RS/P'\cong$

\noindent$  R_Q[X]/(QR_Q[X]+XR_Q[X])\cong R_Q/QR_Q=\kappa (Q)$.

If $P\subseteq Q$, then, $QR_Q[X]+XPR_Q[X]=QR_Q[X]$, so that   

\noindent$\kappa(Q)\otimes_RS/P'\cong R_Q[X]/QR_Q[X]\cong \kappa (Q)[X]$.

Let $a\in R$. If $a$ is a regular element of $R/\!\!/P$, it is regular in $R$. If $a\in P$, then $aX\in P'$, so that $ax=0$, with $x\neq 0$, a contradiction. Conversely, if $a$ is regular in $R$ and $a\not\in P$,  $a$ is a regular element of $R/\!\!/P$. Otherwise, there exists $y\in R/\!\!/P,\ y\neq 0$  such that $ay=0$. We may assume that $y\not\in R$ since $a$ is a regular element of $R$. As in the beginning of the proof, we can write $y=b+xf(x)$ where $b\in R,\ f(X):=\sum_ {i=0}^n\alpha_iX^{i}$ is a polynomial of $R[X],\ n:=\deg(f)$ and either $\alpha_i=0$ or $\alpha_i\not\in P$, with $\alpha_n\not\in P$. Then $ay=0=a(b+xf(x))=ab+axf(x)$ implies $a b+aXf(X)\in P'=XPR[X]$, so that $ab=0\ (iii)$ and $af(X)\in PR[X]\ (iv)$. Since $a$ is a regular element of $R$, it follows that $b=0$ by $(iii)$ and $(iv)$ gives $f(X)\in PR[X]$ because $a\not\in P$. Then, $y=0$, a contradiction, and $a$ is a regular element of $R/\!\!/P$.
 \end{proof}
 
\begin{proposition} \label{Cohn1} If $I$ is a semiprime ideal of a ring $R$, the extension $R\subseteq R/\!\!/I$ is SL, pure and t-closed. Moreover, $I =(R:R/\!\!/I)$ and is also semiprime in $R/\!\!/I$.
If  $R$ is reduced, then $R/\!\!/I$ is reduced. 
\end{proposition} 
\begin{proof} Since $I$ is semiprime, set $I:=\cap_{\lambda\in\Lambda}P_{\lambda}$. Let $P\in\mathrm{Spec}(R)$ be such that $I\subseteq P$, so that $XI[X]\subseteq XP[X]$, and there is a surjective morphism $f:R[X]/XI[X]\to R[X]/XP[X]$ giving the following commutative diagram: 
 $$\begin{matrix}
 R & \overset{i}\to  & R[X]/XI[X]=R/\!\!/I \\
 {} &   j\searrow     &  \downarrow f    \\
 {} &        {}           &  R[X]/XP[X]=R/\!\!/P         
\end{matrix}$$
with two injective morphisms $i$ and $j$. This holds for any $P_{\lambda}$. According to Lemma \ref{Cohn}, $j$ is SL, which implies that $f$ is SL by Proposition \ref{3.1}. It follows that $j(\mathrm U(R))=\mathrm U(R/\!\!/P_{\lambda})=f(\mathrm U(R/\!\!/I))\ (*)$. Let $y\in\mathrm U(R/\!\!/I)$. Then, we may write $y=i(a_0)+\sum_{k=1}^r\overline{a_k}\overline{X}^k$, where $\overline{a_k}$ and $\overline{X}$ are respectively the classes of $a_k\in R$ and $X$ in $R/\!\!/I$. It follows that $f(y)= j(a_0)+\sum_{k=1}^r\tilde{a_k}\tilde{X}^k$, where $\tilde{a_k}$ and $\tilde{X}$ are respectively the classes of $a_k\in R$ and $X$ in $R/\!\!/P_{\lambda}$. We deduce from $(*)$ that $\sum_{k=1}^r\tilde{a_k}\tilde{X}^k=\tilde{0}$, giving $\sum _{k=1}^ra_kX^k\in XP_{\lambda}R[X]$, and then $a_k\in P_{\lambda}\ (**)$ for each $k\geq 1$ and for each $\lambda\in\Lambda$. Moreover, $a_0 \in\mathrm U(R)$. Because $(**)$ holds for each $P_{\lambda}$, we get $a_k\in I$ for each $k\geq 1$ and then $\overline{a_k}\overline{X}^k=\overline{0}$ for each $k\geq 1$. To conclude, $y=i(a_0)\in i(\mathrm U (R))$ and $i(\mathrm U(R))=\mathrm U(R/\!\!/I)$, that is $R\subseteq R/\!\!/I$ is SL.

It is enough to mimic the proofs of Lemma \ref{Cohn} for each $P_{\lambda}$ to get that $ R\subseteq R/\!\!/I$ is pure and t-closed. Moreover, $I=(R:R/\!\!/I)$. Since $R\subseteq R/\!\!/I$ is t-closed, it is also seminormal, so that $I=(R:R/\!\!/I)$ is semiprime in $R/\!\!/I$ by \cite[Lemma 4.8]{DPP2}.

If  $R$ is reduced, then $R/\!\!/I$ is reduced as in Lemma \ref{Cohn}.
 \end{proof}
 
 \begin{corollary} \label{Cohn2} Let $R\subseteq S$ be an extension with $I$ an ideal of $R$ and $K$ a semiprime ideal of $S$ such that $I=R\cap K$. Then, $R\subseteq S$ is SL if and only if $R/\!\!/I\subseteq S/\!\!/K$ is SL.
\end{corollary} 
\begin{proof} Since $I=R\cap K$, we get that $I$ a semiprime ideal of $R$ and $I[X]= R[X]\cap K[X]$, so that the following commutative diagram holds: 
 $$\begin{matrix}
        R         & \subseteq &        S          \\
\downarrow &       {}        & \downarrow \\
   R/\!\!/I       & \subseteq &    S/\!\!/K         
\end{matrix}$$
where the vertical maps are injective and SL morphisms by Proposition \ref{Cohn1}. If either $R\subseteq S$ (1) or $R/\!\!/I\subseteq S/\!\!/K$ (2) is SL, so is $R\to S/\!\!/K$ by Proposition \ref{3.1} (1). If (1) holds, then (2) holds by Proposition \ref{3.1} (3). If (2) holds, then (1) holds by the same Proposition because $S\to S/\!\!/K$ is injective.
\end{proof}
  
 \begin{theorem}\label{5.061} Let $R$ be a ring and $N\in\mathrm{Max}(R[X])$. The following conditions are equivalent:
\begin{enumerate}
\item  $N\cap R\in\mathrm{Max}(R)$.
\item $M[X]\subseteq N$ for some $M\in \mathrm{Max}(R)$. 
\item There exists a monic polynomial $f(X)\in N$. 
 \end{enumerate}
If these conditions hold, then $N=M[X]+f(X)R[X]$ and $M=N\cap R$.
\end{theorem}

\begin{proof} (1)$\Rightarrow$(2) Assume that $M:=N\cap R\in\mathrm {Max}(R)$.
It follows that $M\subseteq N$ which implies $M[X]\subseteq N$ for some $M\in \mathrm{Max}(R)$. 

(2)$\Rightarrow$(3) Let $M\in\mathrm{Max}(R)$ be such that $M[X]\subseteq N$. Of course, $M[X]+R[X]f(X)\subseteq N$ for any $f(X)\in N$. Consider the following commutative diagram: 
$$\begin{matrix}
        R         & \to &           R[X]              \\
\downarrow & {}  &      \downarrow        \\
     R/M       & \to & R[X]/M[X]=(R/M)[X]         
\end{matrix}$$
Since $M\in\mathrm{Max}(R)$, it follows that $R/M$ is a field and $(R/M)[X]$ is a PID. Then, there exists a monic polynomial $f(X)\in N$ such that $N/M[X]=\overline{f(X)}(R[X]/M[X])$, where $\overline{f(X)}$ is the class of $f(X)$ in $R[X]/M[X]$. 

Then any element of $N$ can be written $f(X)g(X)+h(X)$, where $g(X)\in R[X]$ and $h(X)\in M[X]$, giving $N=M[X]+R[X]f(X)$.

(3)$\Rightarrow$(1) Assume that there exists a monic polynomial $f(X)\in N$ and set $f(X):=X^n+\sum_{i= 0}^{n-1}a_iX^{i}\in N$. Consider the following commutative diagram: 
$$\begin{matrix}
          R        & \to &                  R[X]                  \\
 \downarrow & {}  &            \downarrow             \\ 
R/(N\cap R) & \to & R[X]/N=[R/(N\cap R)][x]        
\end{matrix}$$
where $x$ is the class of $X$ in $R[X]/N$. Since $N\in\mathrm{Max}(R[X])$, it follows that $[R/(N\cap R)][x]$ is a field. Moreover, $R/(N\cap R)\to R[X]/N$ is injective and $x^n+\sum_{i= 0}^{n-1}\overline{a_i}x^{i}=0$, where $\overline{a_i}$ is the class of $a_i$ in $ R/(N\cap R)$. This implies that $R[X]/N$ is integral over $R/(N\cap R)$, which gives that $R/(N\cap R)$ is a field by \cite[Theorem 16]{K} because $R/(N\cap R)$ is an integral domain. To conclude, $N\cap R\in\mathrm{Max}(R)$.
If the equivalent conditions (1), (2) and (3) hold then $N=M[X]+f(X)R[X]$ by (2) and $M= N \cap R$.
\end{proof}

\begin{proposition}\label{5.063} Let $R$ be a ring. Then, $R\subseteq R/\!\!/\mathrm J(R)$ is SL, $\mathrm J(R)=(R:R/\!\!/\mathrm J(R))=\mathrm J(R/\!\!/\mathrm J(R))$ and  $R/\!\!/\mathrm J(R)$ is not $\mathrm J$-regular.
\end{proposition}
\begin{proof} The two first assertions come from Proposition \ref{Cohn1} since $\mathrm J (R)$ is semiprime.  

Let $Q\in\mathrm{Max}(R/\!\!/\mathrm J(R))$. Then, there exists $N\in\mathrm{Max}(R [X])$ such that $X\mathrm J(R)[X]\subseteq N$, with $Q=N/(X\mathrm J(R)[X])$. It follows that either $X\in N\ (*)$ or $\mathrm J(R)R[X]\subseteq N\ (**)$. 

In case $(*)$, Theorem \ref{5.061} shows that $M:=N\cap R\in\mathrm{Max}(R)$ and $N =M[X]+XR[X]$. But $\mathrm J(R)\subseteq M$ implies that $\mathrm J(R)[X]\subseteq N$. 
Since $\mathrm J(R)[X]\subseteq N$ in case $(**)$, we get that $\mathrm J(R)[X]\subseteq N$ in both cases. It follows that $N/\mathrm J(R)[X]\in\mathrm{Max}(R[X]/\mathrm J(R)[X])$. 

Conversely, if $N$ is a maximal ideal of $R[X]$ containing $\mathrm J(R)[X]$, it follows that $N/(\mathrm J(R)[X])\in\mathrm{Max}(R[X]/(\mathrm J(R)[X]))$ and $N/(X\mathrm J(R)[X])\in\mathrm{Max}(R[X]/(X\mathrm J(R)[X]))$.

Recall that for a ring $T$, we have $\mathrm J(T[X])=\mathrm{Nil}(T[X])$. Moreover, $R[X]/\mathrm J(R)[X]\cong(R/\mathrm J(R))[X]$ which is reduced. It follows that $\cap[N/\mathrm J(R)[X]\mid\mathrm J(R)[X]\subseteq N$ and $N\in\mathrm{Max}(R[X])]=\mathrm J(R[X]/\mathrm J(R)[X])$

\noindent$=\mathrm J((R/\mathrm J(R))[X])=\overline 0\ (*)$ where $\overline 0$  the class of $0$ in $ R [X]/\mathrm J(R)[X]$. 

To make the reading easier, we set $I:=\mathrm J(R)[X],\ K:=\cap[N\in\mathrm{Max}(R [X])\mid I\subseteq N]$ and $S:=R[X]$. Of course, $I\subseteq K$. We have proved that $\cap[N/I\mid I\subseteq N$ and $N\in\mathrm{Max}(S)]=\overline 0$. Assume that $I\subset K$, so that there exists some $y\in K\setminus I$. This means that $y$ is in any $N\in\mathrm{Max}(S)$ which contains $I$ with $y\not\in I$. If $\overline y$  the class of $y$ in $S/I$, we get that $\overline y\in\cap[N/I\mid I\subseteq N$ and $N\in\mathrm{Max}(S)]=\overline 0$ by $(*)$, giving $y\in I$, a contradiction. Therefore,  $\cap[N\in\mathrm{Max}(S)\mid XI\subseteq N]=I$ which gives $\mathrm J(R/\!\!/\mathrm J (R))=\cap[Q\in\mathrm{Max}(R/\!\!/\mathrm J(R))]=\cap[N/XI\in\mathrm{Max}(R/\!\!/\mathrm J(R))]=I/IX= \mathrm J(R)[X]/X\mathrm J(R)[X]=\mathrm J(R)$. 

At last, $(R/\!\!/\mathrm J(R))/\mathrm J(R)=(R[X]/X\mathrm J(R)[X])/(\mathrm J(R)[X]/X\mathrm J(R)[X])$

\noindent$\cong R[X]/\mathrm J(R)[X]\cong(R/\mathrm J(R))[X]$, which is not zero-dimensional, giving that $R/\!\!/\mathrm J(R)$ is not $\mathrm J$-regular.
\end{proof}

  \section{The ring $R\{X\}$}

We now consider a ring used by Houston and some other authors for results concerning the dimension of rings \cite{HOUST1}. When $R$ is a semilocal (Noetherian) domain, he introduces the ring $R\langle X\rangle$.
 
This notation is in conflict with the notation of the ring used to solve the Serre conjecture. Therefore, we will denote it by $R\{X\}$. Let $\Sigma$ be the multiplicatively closed subset of $R[X]$ defined as follows: let $T_ 1$ be the set of all maximal ideals $N$ of $R[X]$ such that $N\cap R\in\mathrm{Max}(R)$, then $\Sigma$ is the complementary set in $R[X]$ of $\cup[N\mid N\in T_1]$ and $R\{X\}:=R[X]_\Sigma$. As we consider arbitrary rings, a more precise characterization of $T_1$ is given in Theorem \ref{5.061}, completing \cite[Lemma 1.2]{HOUST1}.
  
\begin{proposition}\label{5.14} For any ring $R$, there is a factorization $R\to R\{X\}\to  R(X) $, where $R(X)$ is the Nagata ring of $R$.
 \end{proposition}
\begin{proof} Let $p(X)\in\Sigma$. We claim that $c(p)=R$. Otherwise, there exists $M\in\mathrm{Max}(R)$ such that $c(p)\subseteq M$. Then $ p(X)\in M[X]\subseteq N$ for some $N\in\mathrm{Max}(R[X])$. It follows that $M=N\cap R$, so that $N\in T_1$ by Theorem \ref{5.061} and $p(X)\in N$, a contradiction with $p(X)\in\Sigma$. Since $c(p)=R$, we get that $p(X)/1\in\mathrm U(R(X))$ giving the factorization $R\to R\{X\}\to R(X)$. 
\end{proof}  

We recall that $R$ is a {\it Jacobson ring} if and only if maximal ideals of $R[X]$ contract to maximal ideals of $R$. Then we have the following:

\begin{proposition}\label{5.08} Let $R$ be a Jacobson ring. Then $R\{X\} = R[X]$.
\end{proposition}
\begin{proof} Since $R$ is a Jacobson ring, it follows that $T_1=\mathrm{Max}(R[X])$ where $T_1=\{N\in\mathrm{Max}(R[X])\mid N\cap R\in\mathrm{Max}(R)\}$. This implies that $\Sigma=R[X]\setminus\cup[N\mid N\in T_1]=R[X]\setminus\cup[N\mid N\in \mathrm{Max}(R[X])]=\mathrm U(R[X])$ and $R\{X\}:=R[X]_\Sigma=R[X]$. 
\end{proof}
 
Mimicking \cite{HOUST1} and using also \cite[Lemma 3]{GHe1}, the following proposition characterizes in a more general setting, and when $R$ is $\mathrm J$-regular, the maximal ideals of $R\{X\}$.
 
\begin{proposition}\label{5.11} If  $R$ is $\mathrm J$-regular, $\mathrm{Max}(R\{X\}) = \{NR\{X\}\mid N\in T_1\}$.
\end{proposition}
\begin{proof} Since $R$ is $\mathrm J$-regular, we get that $\mathrm{V}(\mathrm J(R))=\mathrm{Max}(R)$. Let $N\in\mathrm{Max}(R[X])$. Then, $N\in T_1$ if and only if $N\cap R\in\mathrm{Max}(R)$ if and only if $\mathrm J(R)\subseteq N$ if and only if $\mathrm J (R)[X]\subseteq N$. Then, $T_1=\mathrm{Max}(R[X])\cap\mathrm{V}(\mathrm J(R)[X])$. According to \cite[Lemma 3]{GHe1}, and since $R\{X\}=R[X]_{\Sigma}$, it follows that $\mathrm{Max}(R\{X\})=\{NR\{X\}\mid N\in T_1\}$. 
\end{proof}

\begin{proposition}\label{5.12} If $M$ is a maximal ideal of $R$, then there is a factorization $R\to R\{X\} \to R/\!\!/M$ and $R \to R\{X\}$ is faithfully flat.
 \end{proposition}
 
\begin{proof} Since $R\{X\}=R[X]_{\Sigma}$, we have $\mathrm U(R\{X\}) =\{p(X)/q(X)\mid p(X),q(X)\in\Sigma\}$. Let $p(X)\in\Sigma$, so that $p(X)\not\in Q$ for any $Q\in T_ 1$. Let $M\in\mathrm{Max}(R)$. We claim that $p(X)R[X]+XM[X]=R[X]$. Otherwise, there exists some $N\in\mathrm {Max}(R[X])$ such that $p(X)\in N$ and $XM[X]\subseteq N\ (*)$.  But $(*)$ holds if and only if either $X\in N$ (1) or $M[X]\subseteq N$ (2). In case (1), $N$ contains the monic polynomial $X$, so that $N\in T_1$ by Theorem \ref{5.061}, and in case (2), the same reference shows that $N\in T_1$. In both cases, we infer that we get a contradiction with $ N\in T_1$ and $p(X)\in N$ since $p(X)\not\in Q$ for any $Q\in T_1 $, so that $ p(X)R[X]+XM[X]=R[X]$ holds. This implies that there exist $f(X)\in R [X]$ and $g(X)\in M[X]$ such that $p(X)f(X)+Xg(X)=1$. Working in $R/\!\!/M= R[X]/XM[X]$, this gives $\overline{p(X)}\ \overline{f(X)}=\overline{1}$, so that $\overline{p(X)}\in\mathrm U (R/\!\!/M)$. Then, we have the wanted factorization thanks to the following commutative diagram: 
$$\begin{matrix}
R &      \to      &               R[X]              &      \to      & R/\!\!/M=R[X]/XM[X] \\
{} & \searrow &         \downarrow        & \nearrow &             {}               \\ 
{} &      {}       & R\{X\}=R[X]_{\Sigma} &       {}       &             {}
\end{matrix}$$

Consider now the following factorization $R\to R[X]\to R\{X\}\to R/\!\!/M$.
 By Lemma \ref{Cohn}, the extension $R\subseteq R/\!\!/M$ is pure. Then, so is $R\subset R\{X\}$ by \cite[Lemme 2.3, p.19]{OP}. Moreover, $R\subseteq R[X]$ is flat as well as  $R[X]\subseteq R\{X\}$.  Then, so is $R\subset R\{X\} $. Since the maximal ideals of $R$ can be lifted on in $R\{X\}$, then  $R \to R\{X\}$  is faithfully flat.
\end{proof}

 \begin{proposition}\label{5.121} Let $R\subseteq S$ be an integral ring extension. Then, there is an integral ring extension $R\{X\}\subseteq S\{X\}$. In case $S$ is $\mathrm J$-regular, then $R\{X\}\subseteq S\{X\}$ is SL if and only if $R=S$.
\end{proposition}
 
\begin{proof} Set $T_1:=\{N\in\mathrm{Max}(R[X])\mid N\cap R\in\mathrm{Max}(R)\}$, $T' _1:=\{N'\in\mathrm{Max}(S[X])\mid N'\cap S\in\mathrm{Max}(S)\},\ \Sigma:=R[X]\setminus\cup[N\mid N\in T_1]$ and $\Sigma':=S[X]\setminus\cup[N'\mid N'\in T'_1]$. We have the following diagram:
 $$\begin{matrix}
        R         & \to &      R[X]       & \to & R\{X\} \\
\downarrow &  {} & \downarrow & {}  &            \\
        S         & \to &     S[X]        & \to & S\{X\}
\end{matrix}$$
Since $R\subseteq S$ is integral, so is $f:R[X]\subseteq S[X]$. Let $N\in T_1$, so that $ M:=N\cap R\in\mathrm{Max}(R)$, and there exists $N'\in\mathrm{Max}(S[X])$ lying above $N$. Set $M':=N'\cap S\in\mathrm{Spec}(S)$. Then, $M'\cap R=N'\cap S\cap R=N '\cap R[X]\cap S\cap R=N\cap R=M\in\mathrm{Max}(R)$. It follows that $M'\in\mathrm {Max}(S)$ since it is a prime ideal of $S$ lying above a maximal ideal of $R$. Then, $N' \in T'_1$. Conversely, let $N'\in T'_1$ and set $N:=N'\cap R[X]$. Then, $M':=N'\cap S\in\mathrm{Max}(S)$ and $N\cap R=N'\cap R[X]\cap R=N'\cap S\cap R[X]\cap R=M'\cap R\in\mathrm{Max}(R)$ since $M'\in\mathrm{Max}(S)$. It follows that there is a surjective map $T'_1\to T_1$ which is the restriction of ${}^{a}f:\mathrm{Spec}(S[X])\to\mathrm {Spec} (R[X])$. 

Now, let $p(X)\in\Sigma$, so that $p(X)\not\in N$ for any $N\in T_1$. We claim that $p(X)\in\Sigma'$. Otherwise, there exists $N'\in T'_1$ such that $p(X)\in N'$. But $p(X)\in R[X]$, so that $p(X)\in N'\cap R[X]\in T_1$ by the beginning of the proof, a contradiction. Then, $\Sigma\subseteq \Sigma'$ and $\mathrm U(R\{X\})\subseteq\mathrm U(S\{X\})$. 

We show that $f$ defines an injective morphism $R\{X\}\to S\{X\}$. Let $p(X)/q(X)\in R\{X\},\ p(X),q(X)\in R[X],\ q(X)\in\Sigma$ be such that $p(X)/q(X)=0$ in $S\{X\}$. There exist $g(X)\in\Sigma'$ such that $p(X)g(X)=0\ (*)$ in $S[X]$. Then, $g(X)\not\in N'$, for any $ N'\in T'_1$. We claim that $c(g)=S$. Otherwise, there exists $N\in\mathrm{Max}(S)$ such that $c(g)\subseteq N$. Then $g(X)\in N[X]\subseteq N'$ for some $N'\in\mathrm {Max}(S[X])$. It follows that $N'\in T'_1$ by Theorem \ref{5.061} and $g(X)\in N'$, a contradiction with $g(X)\not\in N'$, for any $N'\in T'_1$. Since, $c(g)=S$, we obtain that $g$ is a regular element of $S[X]$ and it results from $(*)$ that $p(X)=0$. Then there is an integral ring extension $R\{X\}\subseteq S\{X\}$.
 
 If $R=S$, obviously $R\{X\}\subseteq S\{X\}$ is SL. 
 
Conversely, assume that $S$ is $\mathrm J$-regular and $R\{X\}\subseteq S\{X\}$ is SL. We mimic the proof of Proposition \ref{3.27}. Let $s\in S$ and set $p(X):=s+X\in S\{X\}$. Since $S$ is $\mathrm J$-regular, then $\mathrm{Max}(S\{X\})=\{NS\{X\}\mid N\in T'_1\}$ by Proposition \ref{5.11}, so that $p(X)\in\mathrm U(S\{X\})=\mathrm U(R\{X\})$. It follows that there exists $h(X),k(X)\in\Sigma$ such that $p(X)/1=h(X)/k(X)$ in $S\{X\}$, and then there exists $q(X)\in\Sigma'$ such that $q(X)p(X)k (X)=q(X)h(X)$. Because $q(X)\in\Sigma'$, we get that $c(q)=S$ (see the proof we give in the previous paragraph), which implies that $q(X)$ is a regular element of $S[X]$. Then $p(X)k(X)=h(X)=(s+X)k(X)$. As $Xk(X)$ and $h(X)$ are in $R[X]$, it follows that $sk(X)\in R[X]$. Now, $R$ is also $\mathrm J$-regular by Corollary \ref{5.27}. The same proof as for $q(X)\in S[X]$ before shows that $c(k)=R$. Set $k(X):=\sum_{i=0}^na_iX^{i},\ a_i\in R$ and $b_i:=sa_i\in R\ (*)$ for each $i\in\{0,\ldots,n\}$. There exist $\lambda_0,\ldots,\lambda_n\in R$ such that $\sum_{i=0}^n\lambda_ia_i= 1$. Multiplying each equality of $(*)$ by $\lambda_i$ and adding each of these equalities for each $i\in\{0,\ldots,n\}$, we get $s(\sum_{i=0}^n\lambda_i a_i)=s=\sum_{i=0}^n\lambda_ib_i\in R$, so that $S=R$.
\end{proof}

\begin{corollary}\label{5.122} Let $R\subseteq S$ be an integral ring extension.  If  $R\{X\}\subseteq S\{X\} $ is SL, so is  $R\subseteq S$.
\end{corollary}
 
\begin{proof} Let $s\in\mathrm U(S)$. In particular, $s/1\in\mathrm U(S\{X\})=\mathrm U(R\{X\})$. Then $s/1=f(X)/g(X)$ in $S\{X\}$ for some $f(X)/g(X)\in\mathrm U(R\{X\})$, that is $f (X),g(X)\in\Sigma$. The same proof as in Proposition \ref{5.121} shows that $sg(X)=f(X)\ (*)$ in $S[X]$ and $c(f)=c(g)=R$. Set $f(X):=\sum_{i=0}^na_iX^{i},a_i\in R$ and $g(X):=\sum_{i=0}^nb_iX^{i},b_i\in R$, so that $sb_i=a_i$ for each $i\in\{0,\ldots,n\}\ (**)$. There exist $\lambda_0,\ldots,\lambda_n\in R$ such that $\sum_{i=0}^n\lambda_ib_i=1$. Multiplying each equality of $(**)$ by $\lambda_i$ and adding each of these equalities for each $i\in\{0,\ldots,n\}$, we get $s(\sum_{i=0}^n\lambda_ib_i)=s=\sum_{i=0}^n\lambda_ia_ i\in R$, so that $s\in R$. Then, $(*)$ shows that $sc(g)=c(f)=R$ and gives that $Rs=R$, that is $s\in\mathrm U(R)$ and $\mathrm U(R)=\mathrm U(S)$. To conclude, $R\subseteq S $ is SL.
\end{proof}

\begin{corollary}\label{5.13} Let $R$ be a $\mathrm J$-regular ring. Then there is a factorization $R \to R\{X\} \to  R/\!\!/\mathrm J(R) $.
 \end{corollary}
\begin{proof} We use the beginning of the proof of Proposition \ref{5.12}. Let $p(X)\in\Sigma $, so that $p(X)\not\in Q$ for any $Q\in T_1$. We claim that $\overline{p(X)}\in\mathrm U(R/\!\!/\mathrm J(R))$. Otherwise, there exists some $N\in\mathrm{Max}(R[X])$ with 

\noindent$X\mathrm J(R)[X]\subseteq N$ such that $p(X)\in N$. As in the quoted proof, $X\not\in N$ because $N\not\in T_1$. It follows that $\mathrm J(R)[X]\subseteq N$. But $R$ being $\mathrm J$-regular, any prime ideal of $R$ containing $J(R)$ is maximal. As $\mathrm J(R)[X]\subseteq N$, we get $\mathrm J(R)\subseteq N\cap R$, so that $N\cap R\in\mathrm{Max}(R)$, giving $N\in T_1$, a contradiction. Then, $\overline{p(X)}\in\mathrm U(R/\!\!/\mathrm J(R))$, and we get the wanted factorization thanks to the following commutative diagram: 
$$\begin{matrix}
R &      \to      &               R[X]              &      \to      & R/\!\!/\mathrm J(R) \\
{} & \searrow &         \downarrow        & \nearrow &             {}               \\ 
{} &      {}       & R\{X\}=R[X]_{\Sigma} &       {}       &             {}
\end{matrix}$$
 \end{proof}
 
\begin{corollary}\label{5.131} Let $R$ be a ring and $I$ an ideal intersection of finitely many maximal ideals $M_i$. Then, $R\to R/\!\!/I$ is SL and there is a factorization $R \to R\{X\} \to  R/\!\!/I\to  R/\!\!/M_i$ for any $M_i$.

If, moreover, $R$ is $\mathrm J$-regular, there is a factorization $R\to R\{X\}\to R/\!\!/\mathrm J(R)\to  R/\!\!/I\to  R/\!\!/M_i$ for any $M_i $.
 \end{corollary}
\begin{proof} By Proposition \ref{Cohn1}, $R\to R/\!\!/I$ is SL. For the other assertions, as we have the inclusions $\mathrm J(R)\subseteq I\subseteq M_i$ for any $M _i$, we get the wanted factorization, mimicking the proof of Corollary \ref{5.13}, and using the fact that $R/I$ is regular as a product of finitely many fields.
\end{proof}

 \section{FCP SL extensions}

\begin{proposition}\label{5.32}Let $R\subseteq S$ be an SL extension where $S$ is a semilocal ring. Then $R$ is semilocal, $\mathrm J(R)=\mathrm J(S)$ and $R\subseteq S$ is an FIP seminormal and infra-integral extension.
\end{proposition} 

\begin{proof} Since $S$ is semilocal, $S$ is $\mathrm J$-regular. According to Remark \ref{5.10} (3), Theorem \ref{5.2} and Proposition \ref{5.3}, we have $\mathrm J(R)=\mathrm J(S)$ and we can assume that $R\subseteq S$ is seminormal and infra-integral, $R$ and $S$ are regular and semi-local and that $S$ is a finite product of fields each of them being a residual field of $R$. Finally the extension can be viewed as a product of extensions $R/M \to (R/M)^n$ where $M\in \mathrm{Max}(R)$. By using suitable localizations and \cite[Proposition 4.15]{DPP3}, we get that $R\subseteq S$ is an FIP seminormal and infra-integral extension.
\end{proof}

\begin{corollary}\label{5.321}Let $R\subseteq S$ be an SL extension where $S$ is a semilocal ring. If $R\subseteq S$ is a flat epimorphism, then $R=S$.
\end{corollary} 

\begin{proof} By Proposition \ref{5.32}, $R\subseteq S$ is integral and a flat epimorphism, then $R=S$.
\end{proof}

\begin{theorem}\label{5.4} An extension $R\subseteq  S$, where $S$ is a semilocal ring, is SL if and only if $R\subseteq S$ is a seminormal infra-integral FIP extension such that $R/M\cong\mathbb Z/2\mathbb Z$ for each $M\in\mathrm{MSupp}(S/R)$.
\end{theorem} 
\begin{proof} Assume that $R\subseteq S$ is SL. According to Proposition \ref{5.32}, we get that $\mathrm J(R)=\mathrm J(S)$ and $R\subseteq S$ is an FIP seminormal and infra-integral extension. Set $I:=J (S)=J(R)=\cap_{i=1}^nM_i=\cap[M_{i,j}\mid M_{i,j}\in\mathrm{Max}(S)]$, where $\mathrm{Max}(R)=\{M_i\}_{i=1}^n$ and $\mathrm{Max}(S)=\{M_ {i,j}\mid i=1,\ldots,n,\ M_ i=M_{i,j}\cap R$ for each $j\in\mathbb N_{n_i}\}$. By Remark \ref{5.10} (3), $R/I\subseteq S/I$ is SL. But $R/I\cong\prod_{i= 1}^nR/M_i$ and $S/I\cong\prod[S/M_{i,j}\mid M_{i,j}\in\mathrm{Max}(S)]$. Set $X_1:=\mathrm{MSupp}(S/R),\ X_2:=\mathrm{Max}(R)\setminus X_ 1,\ I_1:=\cap_{M\in X_1}M,\ I_2:=\cap_{M\in X_2}M,\ R_1:=R/I_1=\prod_ {M\in X_1}R/M$ and $R_2:=R/I_2=\prod_{M\in X_2}R/M$. Then, $R/I=R_ 1\times R_2$ because $I_1$ and $I_2$ are comaximal. It follows that $\mathrm U(R/I)=\mathrm U(R_1)\times\mathrm U(R_2)$. For $M_i\in X_2 $, there is a unique $M_{i,j}$ lying above $M_i$ and we have $R/M_i\cong S/M_{i,j}$, so that $R_2\cong S_2:=\prod_{M_i\in X_2}S/M_{i,j}$ and $\mathrm U(R_2)\cong\mathrm U(S_2)\ (*)$. Set $S_1:=\prod_{M_i\in X_1}S/M_{i,j}=\prod_{M_i\in X_1}[\prod_{j\in\mathbb N_i}S/M_{i,j}]$. Since $R\subseteq S$ is infra-integral and $M_i=M_{i,j}\cap R$ for each $j\in\mathbb N_{n_i}$, we get that $R/M_i\cong S/M_{i,j}$, so that $S_1 \cong\prod_{M_i\in X_1}(R/M_i)^{n_i}$. Whence $\mathrm U(S_1)\cong\prod_{M_i\in X_1}[\mathrm U(R/M_i)]^{n_i}\ (**)$. Since $R\subseteq S$ is SL and because of $(*)$, we get that $\mathrm U(R_1)\cong\mathrm U (S_1)\cong\prod_{M_i\in X_1}[\mathrm U(R/M_i)]^{n_i}$. But, for any $M_ i\in X_1$, we have $R_{M_i}\subset S_{M_i}$ seminormal and infra-integral, which gives that $n_i>1$ for any $M_i\in X_1$. Then, $\mathrm U(R_1)\cong\mathrm U(S_1)$ if and only if $|\mathrm U(R/M_i)|=1$ for any $M_i\in X_1$. Since $R/M_i$ is a field for any $M _i\in X_ 1$, it follows that $R/M_i\cong\mathbb Z/2\mathbb Z$ for each $M_i\in\mathrm{MSupp}(S/R)$.

Conversely, assume that $S$ is a semilocal ring and $R\subseteq S$ is a seminormal   infra-integral FIP extension such that $R/M\cong\mathbb Z/2\mathbb Z$ for each $M\in \mathrm{MSupp}(S/R)$. We can use the results and notation of the first part of the proof related to the sets $I_1,I_2,\ R_1,R_2$ and $S_1,S_2$. Since $R\subseteq S$ is  seminormal FIP with $S$ semilocal, we get that $(R:S)=\cap_{M_i\in X_1}M_i=\cap_{M_i \in X_1}M_{i,j}$ by \cite[Proposition 2.4]{SPLIT}. In addition, $\mathrm J(S)=\cap_{i=1}^n M_{i,j}=(R:S)\cap_{M_i\in X_2}M_{i,j}=(R:S)\cap_{M_i\in X_2}M_i=\mathrm J(R)$. Set $ I:=J(S)=J(R)$. We still have $\mathrm U(R/I)=\mathrm U(R_1)\times\mathrm U(R_2),\ \mathrm U(R_2)\cong\mathrm U(S_2),\ \mathrm U(S/I)=\mathrm U(S_1)\times\mathrm U (S_2)$ and $U(S_1)\cong\prod_{M_i\in X_1}[\mathrm U(R/M_i)]^{n_i}$. Since $R/M_i\cong\mathbb Z/2\mathbb Z$ and $R/M_i\cong S/M_{i,j}$ for any $M_i\in X_1$ and $M_ {i,j}$ lying over $M_i$, it follows that $|\mathrm U(R/M_i)|=1=|\mathrm U(S/M_{i,j})|$ for any $M_i\in X_1$, so that $\mathrm U(R_1)\cong\mathrm U(S_1)$, which yields that $R/I\subseteq S/I$ is SL and $R\subseteq S$ is SL by Corollary \ref{5.1}.
 \end{proof}
 
\begin{corollary}\label{5.5} A seminormal infra-integral FCP extension $R\subseteq S$ is SL if  $R/M\cong\mathbb Z/2\mathbb Z$ for each $M\in\mathrm{MSupp}(S/R)$.
  \end{corollary}
\begin{proof} Since $R\subseteq S$ is an FCP extension, it follows that $\mathrm {MSupp}(S/R)$ has finitely many elements. Set $I:=(R:S)$. Moreover, $R\subseteq S$ being seminormal infra-integral, $I$ is an intersection of finitely many maximal ideals of $ S$ by \cite[Proposition 2.4]{SPLIT}. It follows that the extension $R/I\subseteq S/I$ satisfies the assumptions of Theorem \ref{5.4} because $R/I\subseteq S/I$ has FIP by \cite[Proposition 4.15]{DPP3}, and is then SL. Now, it is enough to use Corollary \ref{5.1} to get that $R\subseteq S$ is SL.
\end{proof}

The following result did not seem to appear in earlier papers: we show that there is a seminormal infra-integral closure for FCP extensions. This will be useful to build a closure for SL extensions in an FCP extension.

\begin{lemma}\label{5.51} Let $R\subseteq S$ be an FCP extension and $T,V\in[R,S]$ be such that $R\subseteq T$ and $R\subseteq V$ are both seminormal infra-integral. Then, $R\subseteq TV$ is seminormal infra-integral.
  \end{lemma}
\begin{proof} Since $R\subseteq T$ and $R\subseteq V$ are both  infra-integral, we have $T,V\in[R,{}_S^tR]$, giving $TV\in[R,{}_S^tR]$, so that $R\subseteq TV$ is infra-integral. 

We claim that $R\subseteq TV$ is seminormal. Let $M\in\mathrm{MSupp}(TV/R)$, so that $(TV)_M=T_MV_M\neq R_M$. In particular, either $T_M\neq R_M$ or $V_M\neq R _M$. In both cases, $R_M\subseteq T_M$ and $R_M\subseteq V_M$ are either seminormal or equality. Set $R':=R_M,\ T':=T_M,\ V':=V_M$ and $M':=MR_M$. According to \cite[Proposition 2.4]{SPLIT}, we get that $M'\subseteq(R':T')$ and $M' \subseteq(R':V')$, giving $M'T'=M'V'=M'\subseteq R'$, so that $M'T'V'=M'\subseteq R'$. It follows that $M'=(R':T'V')$. Set $M'=\cap_{i=1}^nM'_i=\cap_{j=1}^mN'_j\ (*)$, where the $M'_i$ are in $\mathrm{Max}(T')$ (resp. $N'_j$ are in $\mathrm{Max}(V')$) because $R' \subseteq T'$ and $R'\subseteq V'$ are either seminormal or equality. It follows that $T'/M'_i\cong R'/M'\cong V'/N'_j$ for each $i$ and $j$ since $R'\subseteq T'$ and $R' \subseteq V'$ are both infra-integral. In addition, $T'/M'\cong(R'/M')^n$ and $V'/M'\cong (R'/M')^m$ by $(*)$. This implies that $(T'V')/M'\cong(T'/M')(V'/M')\cong(R'/M')^{n+m}$ where $R'/M'$ is a field. According \cite[Proposition 4.15]{DPP3}, we get that $R'/M' \subseteq(T'V')/M'$ is seminormal infra-integral as $R'=R_M\subseteq T'V'=T_MV_M=(T V)_M$. Since this holds for any $M\in\mathrm{MSupp}(TV/R)$, we get that $R\subseteq TV$ is  seminormal infra-integral.  
\end{proof}
 
\begin{proposition}\label{5.52} Let $R\subseteq S$ be an FCP extension. There exists a greatest $T\in[R,S]$ such that $R\subseteq T$ is seminormal infra-integral. It satisfies the following properties:
 \begin{enumerate}
\item $T=\Pi[V\in[R,S]\mid R\subseteq V$ seminormal infra-integral$]$.
\item $T=\sup[V\in[R,S]\mid R\subseteq V$ seminormal infra-integral$]$.
\item $T=\cup[V\in[R,S]\mid R\subseteq V$  seminormal infra-integral$]$.
 \end{enumerate}
 \end{proposition}
\begin{proof} Set $\mathcal F:=\{V\in[R,S]\mid R\subseteq V$ seminormal infra-integral$\}$. The result is obvious if $\mathcal F=\{R\}$. So, assume that $\mathcal F\neq\{R\}$. Since $R\subseteq V$ is infra-integral for any $V\in\mathcal F$, it follows that $\mathcal F\subseteq[R,{}_S^tR]$. In particular, any $V\in\mathcal F$ is integral over $R$. Then, we may assume that $R\subseteq S$ is infra-integral. Since $R\subseteq S$ has FCP, $\mathcal F$ has maximal elements. We claim that $\mathcal F$ has only one maximal element. Otherwise, there exist $V,V',\ V\neq V'$ which are maximal elements of $\mathcal F$. Then, $VV'\not\in\mathcal F$. According to Lemma \ref{5.51}, we get that $ R\subseteq VV'$ is seminormal infra-integral, a contradiction. Then, $\mathcal F$ has only one maximal element. Let $T$ be this maximal element. Equalities (1), (2) and (3) follow obviously because $V\subseteq T$ for any $V\in\mathcal F$ and $T\in\mathcal F$.  
\end{proof}

There are four types of minimal extensions, but only two types are used in the paper, characterized in \cite [Theorems 2.1 and 2.2]{DPP2} and \cite[Proposition 4.5]{Pic 6}. 
We recall some results about minimal extensions: 

\begin{proposition}\label{5.521} (1) \cite[Theorem 2.1]{DPP2} Let $R\subseteq T$ be a minimal integral extension. Then, $M:=(R:T)\in\mathrm{Max}(R)$ and there is a bijection $\mathrm{Spec}(T)\setminus\mathrm{V}_T(M)\to\mathrm{Spec}(R)\setminus\{M\}$, with at most two maximal ideals of $T$ lying above $M$.

(2) Let $R\subseteq S$ be an FCP extension. Then, any maximal chain of $[R,S]$ results from juxtaposing finitely many minimal extensions.
 \end{proposition}
 \begin{proof} Obvious.
 \end{proof}
 
\begin{definition}\label{5.06} (1) Let $R\subset T$ be an extension and $M:=(R: T)$. Then $R\subset T$ is minimal {\it decomposed} if and only if $M\in\mathrm{Max}(R)$ and there exist $M_1,M_2\in\mathrm{Max}(T)$ such that $M=M_1\cap M_2$ and the natural maps $R/M\to T/M_i$ for $i\in\{1,2\}$ are both isomorphisms. A minimal decomposed extension is seminormal infra-integral.

(2) A minimal extension $R\subset T$ is minimal {\it Pr\"ufer} if and only if $R\subset T$ is a flat epimorphism and there exists $M\in\mathrm{Max}(R)$ such that $MT=T$ with $ R_P\cong T_P$ for any $P\in\mathrm{Spec}(R),\ P\neq M$ \cite[Theorem 2.2]{FO}. In particular, there is a bijection $\mathrm{Spec}(T)\to\mathrm{Spec}(R)\setminus\{M\}$. 
\end{definition}
 
\begin{lemma}\label{5.04} Let $R\subseteq S$ be an FCP integral extension  where $S$ is semilocal. 
 \begin{enumerate}
\item Let $T\in]R,S]$ be such that $R\subset T$ is minimal. Then, $R\subset T$ is SL if and only if $R\subset T$ is minimal decomposed such that $M:=(R:T)$ satisfies $R/M\cong\mathbb Z/2\mathbb Z$.
\item Let $T\in]R,S]$. Then $R\subset T$ is SL if and only if $T\in[R,{}_S^uR]$ and $T/(R:T)\cong (\mathbb Z/2\mathbb Z)^m$ for some positive integer $m$.
\item Let $T,V\in[R,S]$ be such that $R\subset T$ and $R\subset V$ are SL. Then, $R\subset TV$ is SL.
 \end{enumerate}
 \end{lemma}
\begin{proof} Since $R\subseteq S$ is an FCP integral extension such that $S$ is semilocal, any element of $[R,S]$ is semilocal. 

(1) Let $T\in[R,S]$ be such that $R\subset T$ is minimal SL. Then, $T$ is semilocal and $R\subset T$ satisfies conditions of Theorem \ref{5.4}, that $R\subset T$ is a seminormal infra-integral extension such that $R/M\cong\mathbb Z/2\mathbb Z$ for $M= (R:T)$ since $\mathrm{MSupp}(T/R)=\{M\}$. It follows that $R\subset T$ is minimal decomposed by \cite[Proposition 4.5]{Pic 6}.

Conversely, if $R\subset T$ is minimal decomposed such that $M:=(R:T)$ satisfies $R/M\cong\mathbb Z/2\mathbb Z$, then $R\subset T$ is seminormal infra-integral with $R/M\cong\mathbb Z/2\mathbb Z$ for $M\in\mathrm{MSupp}(T/R)$ by the previous reference. Then Theorem \ref{5.4} shows that $R\subset T$ is   SL.

(2) According to Theorem \ref{5.4}, $R\subseteq T$ is SL if and only if (i) and (ii)   hold where :

(i) $R\subseteq T$ is a seminormal infra-integral FIP extension.

(ii)  $R/M\cong \mathbb Z/2\mathbb Z$ for each $M\in \mathrm{MSupp}(T/R)$.

Assume first that $R\subseteq T$ is SL. By (i) and \cite[Theorem 5.9]{Pic 15}, we get that $R\subseteq T$ is u-integral, so that $T\in[R,{}_S^uR]$. Moreover, since $R\subseteq T$ is seminormal, it follows that $(R:T)=\cap[N_{i,j}\mid N_{i,j}\in\mathrm{Max}(T),\ N_{i,j}\cap R=M_i]$, where $\mathrm{MSupp}(T/R):=\{M_i\}_{i=1}^n$. Then, $T/(R: T)\cong T/(\cap N_{i,j})\cong\Pi(T/N_{i,j})\cong\Pi(R/M_i)^{n_i}$, where $n_i:=|\{N_{i,j}\in\mathrm{Max}(T)\mid N_{i,j}\cap R=M_i\}|$ and because of (i). Then, (ii) gives that $T/(R:T)\cong(\mathbb Z/2\mathbb Z)^m$ for some positive integer $m$.

Conversely, assume that $T\in[R,{}_S^uR]$ and $T/(R:T)\cong(\mathbb Z/2\mathbb Z)^ m$ for some positive integer $m$. Since $T\in[R,{}_S^uR]$, we get that $R\subset T$ is u-integral, and in particular infra-integral. Moreover, since $T/(R:T)\cong(\mathbb Z/2\mathbb Z)^m$, a product of 
 finitely many finite fields, $T/(R:T)$ is reduced, so that $(R:T)$ is an intersection of finitely many maximal ideals of $T$ (and $R$), and $R\subset T$ is seminormal by \cite[Proposition 2.4]{SPLIT}. Then, (i) holds
 because $T/(R:T)$ has finitely many elements, giving that $R\subseteq T$ has FIP. At last, $(R:T)$ is semi-prime, it is an intersection of the maximal ideals $N_{i,j}$ of $T$ lying above the maximal ideals $M_i$ of $R$ of $\mathrm{MSupp}(T/R)$. It follows that $R/M\cong\mathbb Z/2 \mathbb Z$ for each $M\in\mathrm{MSupp}(T/R)$ and (ii) holds. To conclude, $R\subseteq T$ is SL (we also can use Corollary \ref{5.5}).

(3) Let $T,V\in[R,S]$ be such that $R\subset T$ and $R\subset V$ are SL. Then $R\subset T$ and $R\subset V$ are both seminormal infra-integral, and so is $R\subset TV$ by Lemma \ref{5.51}. Moreover, $\mathrm{MSupp}(TV/R)=\mathrm{MSupp}(T/R)\cup\mathrm{MSupp}(V/R)$, so that any $M\in\mathrm{MSupp}(TV/R)$ is either in $\mathrm {MSupp}(T/R)$ or in $\mathrm{MSupp}(V/R)$, which implies that $R/M\cong\mathbb Z/2\mathbb Z$, whence $R\subseteq TV$ is SL by Corollary \ref{5.5}.
 \end{proof}
 
\begin{theorem}\label{5.05} Let $R\subseteq S$ be an FCP extension where $S$ is semilocal. There exists a greatest $T\in[R,S]$ such that $R\subseteq T$ is SL. It satisfies the following properties:
 \begin{enumerate}
\item $T=\Pi[V\in[R,S]\mid R\subseteq V$ seminormal infra-integral, $R/M\cong\mathbb Z/2\mathbb Z$ for any $M\in\mathrm{MSupp}(V/R)]$.
\item $T=\Pi[V\in[R,S]\mid R\subseteq V$  SL$]$.
\item $T=\sup[V\in[R,S]\mid R\subseteq V$  SL$]$.
\item $T=\cup[V\in[R,S]\mid R\subseteq V$  SL$]$.
 \end{enumerate}
 \end{theorem}
\begin{proof} 
 Let $V\in[R,S]$ be such that $R\subseteq V$ is SL. According to Lemma \ref{5.04}, we have $V\in[R,{}_S^uR]$ and $R\subset V$ is u-integral, and in particular $ V\in[R,\overline R]$, where $\overline R$ is the integral closure of $R\subseteq S$. As $ \overline R\subseteq S$ has FCP, any element of $[\overline R,S]$ is semilocal. It follows that we can assume that $R\subseteq S$ is an FCP integral extension such that $S$ is semilocal.
 
Setting $\mathcal F':=\{V\in[R,S]\mid R\subseteq V$ seminormal infra-integral, $R/M\cong\mathbb Z/2\mathbb Z$ for any $M\in\mathrm{MSupp}(V/R)]\}$ and using notation of Proposition \ref{5.52}, we get that $\mathcal F'\subseteq\mathcal F$ and $V\in\mathcal F'$ if and only if $R\subseteq V$ is SL by Theorem \ref{5.4}, because $R\subseteq V$ has FIP (see the proof of Lemma \ref{5.04}(2). Since $R\subseteq S$ has FCP, $\mathcal F'$ has maximal elements. Applying Lemmas \ref{5.51} and \ref{5.04}, we get that $\mathcal F'$ has only one maximal element $ T$ which is the greatest $V\in[R,S]$ such that $R\subseteq V$ is SL. Equations (1), (2), (3) and (4) follow obviously because $V\subseteq T$ for any $V\in\mathcal F'$ and since  $T\in\mathcal F'$.  
\end{proof}
 
\section{Some special cases of SL extensions}
 
In this last section, we characterize SL extensions satisfying another property at the same time. We recall that an extension $R\subseteq S$ is called {\it Boolean} if $[R,S]$ is a Boolean lattice, that is a distributive lattice such that each $T\in[R,S]$ has a complement $T'$ in $[R,S]$ (such that $T\cap T'=R$ and $TT'=S$) \cite{Pic 10}. 

\begin{corollary}\label{Bool} Let $R\subseteq S$ be an FCP SL extension  such that $S$ is semilocal and $|\mathrm{V}_S(MS)|=2$ for each $M\in\mathrm{MSupp}(S/R)$. Then $R\subseteq S$ is a Boolean extension.
 \end{corollary}
\begin{proof} Since $R\subseteq S$ is an SL extension such that $S$ is semilocal, Proposition \ref{5.32} gives that $R\subseteq S$ is a seminormal infra-integral FIP extension, and so is $R_M\subseteq S_M$ for each $M\in\mathrm{MSupp}(S/R)$. But $|\mathrm{V}_S(MS)|=2$ for each $M\in\mathrm{MSupp}(S/R)$ implies that $R_M\subseteq S_M$ is minimal decomposed by \cite[Lemma 5.4]{DPP2} and then Boolean by \cite[Lemma 3.27]{Pic 10}, from which we can infer that $R\subseteq S$ is a Boolean extension by \cite[Proposition 3.5]{Pic 10}.
\end{proof}

\begin{remark}\label{5.6} The SL criteria of Theorem \ref{5.4} (resp. Corollary \ref{5.5}) may hold even if $S$ (resp. $S/(R:S)$) is not a semilocal ring. 
 
A weaker condition is gotten with the following example. Take for instance $R:=\prod_{i\in\mathbb N}R_i$ where $R_i\cong\mathbb Z/2\mathbb Z$ for each $i\in\mathbb N$ and set $S:=R^2$. According to \cite[Proposition 1.4]{Pic 9}, $R\subseteq S$ is  seminormal infra-integral but not FCP with $(R:S)=0$, so that $S/(R:S)$ is not semilocal. But since $|\mathrm U(R)|=|\mathrm U(S)|=1$,  $R\subseteq S$ is SL.
\end{remark}

\begin{example}\label{5.7} (1) In Number Theory, we can find a lot of SL extensions $R\subseteq S$ that are not FCP and $S$ is neither semilocal nor regular. Let $K:=\mathbb Q(\sqrt{-d})$, where $d$ is a square-free positive integer such that $d\neq 1,3$. By \cite[10.2, {\bf D}, p.169]{Ri}, $\mathrm U(A)=\{1,-1\}$, where $A$ is the ring of integers of $K$. As $\mathrm U(\mathbb Z)=\{1,-1\}$, we get that $\mathrm U (\mathbb Z)=\mathrm U(R)=\mathrm U(A)$, for any $R\in[\mathbb Z,A]$, so that $\mathbb Z\subseteq R$ and $R\subseteq A$ are SL for any $R\in[\mathbb Z, A]$. Of course, $\mathbb Z\subseteq A$ is an integral extension which has not FCP since $(\mathbb Z:A)=0$ \cite[Theorem 4.2]{DPP2} and $A$ is neither semilocal nor  regular. This also holds for a ring of integers $R$ with integral closure $A$. By Dirichlet's Theorem \cite[Theorem 1, page 179]{Ri}, these are the only cases where the ring of algebraic integers $A$ of an algebraic number field is such that $\mathbb Z\subseteq A$ is SL.

(2) We can say more considering a particular situation of (1). Let $d:=5$. Then, $A=\mathbb Z+\mathbb Z\sqrt{-5}$ by \cite[5.4, {\bf V}, p.97]{Ri} and 5 is ramified in $K$ \cite[11.2, {\bf K}, p.199]{Ri}. This means that $5A=M^2$, where $M\in\mathrm{Max}(A)$. Set $P:=5A$ and $R:=\mathbb Z+5A$. Then, $R\in[\mathbb Z,A]$ with $R\subseteq A$ SL by (1). We have $P=(R:A)\in\mathrm{Max}(R)$ because $R/P=(\mathbb Z+5A)/5A\cong\mathbb Z/(\mathbb Z\cap 5A)=\mathbb Z/5\mathbb Z$ which is a field. Since $P= M^2$, it follows that $M$ is the only maximal ideal of $A$ lying above $P$. Then, $A_P$ is a local ring. Assume that $R_P\subseteq A_P$ is SL. Then, Example \ref{exam}(8) says that $R_P=A_P$, a contradiction with $P=(R:A)\in\mathrm{MSupp}(A/R)$. 

This example shows that Proposition \ref{3.24} has no converse: an extension $R\subseteq S$ may be SL with $R_M\subseteq S_M$ not SL for some $M\in\mathrm{MSupp}(S/R)$. 
\end{example}

Let $R$ be a ring and $F$ its {\it prime subring}, that is $F:=\mathbb Z/\ker(c)$, where $c:\mathbb Z\to R$ is the ring morphism defined by $c(x):=x1_R$. Then, either  $F=\mathbb Z$ or $F=\mathbb Z/n\mathbb Z$ where $n$ is the least positive integer such that $n1_R=0_R$. 

\begin{proposition}\label{6.1} Let $R$ be a ring with $F$ its prime subring. Then $F[\mathrm U(R)]\to R$  is SL.  
\end{proposition}
\begin{proof} Use Example \ref{exam}(1a).
\end{proof}

\begin{corollary}\label{6.2} If $R$ is a local ring with prime subring $F$, then $F[\mathrm U(R)]=R$. If, in addition, $\mathrm U(R)$ is a finitely generated group, $R$ is Noetherian. 
\end{corollary}
\begin{proof} By Proposition \ref{6.1} and Example \ref{exam}(8) $F[\mathrm U(R)]= R$.

Now, assume that $\mathrm U(R)$ is a finitely generated group. Let $\{x_1,\ldots,x_n\}$ be a system of generators of $\mathrm U(R)$, so that any $x\in\mathrm U(R)$ is of the form $x=\prod_{i=1}^nx_i^{n_i}$, with $n_i\in\mathbb Z$ for each $i\in\mathbb N_n$. Then, $R=F[\mathrm U(R)]$ is an $F$-algebra generated by $\{x_i,x_i^{-1}\mid i\in \mathbb N_n\}$, where either $F=\mathbb Z$ or $F=\mathbb Z/m\mathbb Z$, and so is Noetherian.
\end{proof}

\begin{proposition}\label{6.3} Let $S$ be a ring with prime subring $F$.
\begin{enumerate}
\item Let $\Sigma$ be a saturated multiplicatively closed subset of $S$ and set $R:=F[\Sigma]$. Then, $R_{\Sigma}\subseteq S_{\Sigma}$ is SL and so is  $R\subseteq S$. 

If, in addition, $S_{\Sigma}$ is regular, then $R_{\Sigma}\subseteq S_{\Sigma}$ is u-integral, infra-integral,  seminormal, quadratic and $R_{\Sigma}$ is regular.
\item Assume that $\Sigma$ is the set of regular elements of $S$.
\begin{enumerate}
\item Then $R_{\Sigma}\subseteq \mathrm{Tot}(S)$ is SL.
\item If, in addition, $\mathrm{Tot}(S)$ is regular, so is $R_{\Sigma}$.
\item If $S$ has few zerodivisors, then $\mathrm{Tot}(S)$ is semilocal as $R_{\Sigma}$.
\end{enumerate}
\end{enumerate}
 \end{proposition}
\begin{proof} (1) Let $x=a/s\in S_{\Sigma}$, with $a\in S$ and $s\in\Sigma$. Then, $x\in \mathrm U(S_{\Sigma})\Leftrightarrow$ there exists $y=b/t\in S_{\Sigma}$, with $b\in S$ and $t\in\Sigma$ such that $xy=1\ (*)$. Now, $(*)\Leftrightarrow ab/st=1\Leftrightarrow$ there exists $u\in\Sigma$ such that $uab=ust\in\Sigma$. It follows that $a\in\Sigma$ and $\mathrm U(S_{\Sigma})=\{a/s\in S_{\Sigma}\mid a,s\in\Sigma\}$. Then, $\mathrm U(S_ {\Sigma})\subseteq R_{\Sigma}$, with obviously $\mathrm U(S_{\Sigma})\subseteq\mathrm U(R_{\Sigma})\subseteq\mathrm U(S_{\Sigma})$ giving $\mathrm U(R_ {\Sigma})=\mathrm U(S_{\Sigma})$. Then $R_{\Sigma}\subseteq S_{\Sigma}$ is  SL, and so is  $R\subseteq S$ by Proposition \ref{3.25}. 

Assume, in addition, that $S_{\Sigma}$ is regular. Then, according to Proposition \ref{5.3}, $R_{\Sigma}\subseteq S_{\Sigma}$ is u-integral, infra-integral, seminormal, quadratic and $R_{\Sigma}$ is regular.

(2) Now, $\Sigma$ is the set of regular elements of $S$. Then $\mathrm{Tot}(S)=S_ {\Sigma} $. 

(a) and (b) follow from (1). 

(c) If $S$ has few zerodivisors, then $\mathrm Z(S)=\cup_{i=1}^nP_i$ is a finite union of prime ideals of $S$ and $\Sigma=S\setminus\mathrm Z(R)$ gives that $S_{\Sigma}=\mathrm{Tot}(S)$ has finitely many maximal ideals, so that $\mathrm{Tot}(S)$ is semilocal as $R_{\Sigma}$ since $R_{\Sigma}\subseteq S_{\Sigma}$ is integral. 
\end{proof}
 
In \cite{AC}, D. D. Anderson and S. Chun introduced strongly inert extensions 
 and related extensions in the following way: a ring extension $R\subseteq S$ is {\it strongly inert} (resp. {\it weakly strongly inert}) if for nonzero $a,b\in S$, then $ab\in R$ (resp. $ab\in R\setminus \{0\}$) 
implies $a,b\in R$. 

The link between strongly inert and SL extensions has been noticed by Anderson and  Chun and gives many examples of SL extensions.
 
 \begin{proposition} \label{7.1} Let $R\subseteq S$ be a strongly inert extension. The following properties hold:
 \begin{enumerate}
\item\cite[Proposition 3.1 (2) and Theorem 3.2]{AC} $R\subseteq S$ is SL and either $R=S$ or $R$ and $S$ are integral domains.
 \item If $R\neq S$, then  $0$ is the only proper ideal  shared by $R$ and $S$.
 \item If $S$ is $\mathrm J$-regular, then $R=S$.
  \end{enumerate}
 \end{proposition}
\begin{proof} (2) Assume that $R\neq S$, so that $R$ and $S$ are integral domains. Let $I$ be a proper ideal shared by $R$ and $S$. We claim that $I=0$. Otherwise, there exists some $a\in I,\ a\neq 0$. Since $R\subseteq S$ is SL and $R\neq S$, there exists some $x\in S\setminus R$ with $x\neq 0 $. Then, $ax\in I\subseteq R$. It follows that $a x$ is in $R$. Since $R\subseteq S$ is strongly inert, this implies $x\in R$, a contradiction. Then, $0$ is the only proper ideal shared by $R$ and $S$. 
 
(3) Since $R\subseteq S$ is strongly inert, $R\subseteq S$ is SL. If, moreover, $S$ is $\mathrm J$-regular, then, by Theorem \ref{5.2}, $\mathrm J(R)=\mathrm J(S)$ and $R\subseteq S$ is integral. But, \cite[Theorem 3.5 (7)]{AC} says that any element of $S\setminus R$ is transcendental over $R$. It follows that $R=S$.   
  \end{proof}
 
\begin{remark} \label{7.2} The strongly local property is weaker than the strongly inert property in the following way: Let $R\subseteq S$ be SL and $a,b\in S$ be such that $ab\in\mathrm U(R)$. Then $ab\in\mathrm U(S)$ which implies that $a,b\in\mathrm U(S)=\mathrm U(R)$, so that $a,b\in R$.

But, in case $R$ is a field, the following Corollary shows that the notions of SL extension and of weakly strongly inert extension are equivalent.
\end{remark}

\begin{corollary} \label{7.8} Let $K\subseteq S$ be an extension where $K$ is a field. The following properties holds:
\begin{enumerate}
\item $K\subseteq S$ is weakly strongly inert if and only if $K\subseteq S$ is SL.
\item If, in addition, $S$ is an integral domain, then $K\subseteq S$ is strongly inert if and only if $K\subseteq S$ is SL. In this case,  $K\subseteq S$ is algebraically closed.
\item If $K\subseteq S$ is SL, then $K\subseteq S$ is 
 an FCP extension if and only  if either $K=S$ or 
$K\cong\mathbb Z/2\mathbb Z$ and $S\cong K^n$ for some integer $n$. 
\end{enumerate}
\end{corollary}
\begin{proof} Since $\mathrm U(K)=K\setminus \{0\}$, we get that $K\subseteq S$ is SL if and only if $\mathrm U(S)=K\setminus \{0\}$.

(1) Assume that $K\subseteq S$ is SL and let $a,b\in S\setminus\{0\}$ be such that $ab\in K\setminus\{0\}$. Then, $ab\in\mathrm U(S)$, so that $a,b\in\mathrm U(S)=K\setminus\{0\}\subseteq K$ and $K\subseteq S$ is weakly strongly inert. Conversely, if $ K\subseteq S$ is weakly strongly inert, then $K\subseteq S$ is SL according to \cite[Theorem 4.1]{AC}.

(2) Assume that $K\subseteq S$ is SL and, in addition, that $S$ is an integral domain. It follows that $K\subseteq S$ is weakly strongly inert by (1) and $\mathrm Z(S)=\mathrm Z (K)=\{0\}$, so that $K\subseteq S$ is strongly inert by \cite[Proposition 3.1]{AC} and algebraically closed by \cite[Theorem 3.5]{AC}. Conversely, if $K\subseteq S$ is  strongly inert, then $K\subseteq S$ is SL according to \cite[Proposition 3.1]{AC}.

(3) Assume that $K\subseteq S$ is SL.

 If $K=S$, then $K\subseteq S$ has FCP.

If $K\cong\mathbb Z/2\mathbb Z$ and $S\cong K^n$ for some integer $n$, then $K\subseteq S$ has FCP by \cite[Proposition 1.4]{Pic 9} or because $S$ has finitely many elements.

Conversely, assume that $K\subseteq S$ has FCP. Then $S$ is semilocal by Proposition \ref{5.521} and Definition \ref{5.06}. It follows from Theorem \ref{5.4} that $K\cong\mathbb Z/2\mathbb Z$ because $\mathrm{MSupp}(S/K)=\{0\}$ and $S\cong K^n$ for some integer $n$ by Lemma \ref{5.04} since $K\subseteq S$ is infra-integral, and then integral. 
\end{proof}

We say that an extension $R\subseteq S$ is {\it semi-inert} if for nonzero $a,b\in S,\ ab\in R$ implies either $a\in R$ or $b\in R$. Such an extension exists by \cite[Proposition 3.1]{FO}, for example a minimal Pr\"ufer extension. 
 
More generally, if $R\subseteq S$ is semi-inert and $S\subseteq T$ is strongly inert, then $R\subseteq T$ is  semi-inert.
 
\begin{proposition} \label{7.3} Let $R\subseteq S$ be a semi-inert extension. Then $R\subseteq S$ is integrally closed. If, in addition, $R\subseteq S$ is local, then $R\subseteq S$ is SL.
\end{proposition}
   
\begin{proof} Let $s\in S\setminus R$ and assume that $s$ is integral over $R$. Then $s^n+\sum_{i= 0}^{n-1}a_is^{i}=0\ (*)$ for some positive integer $n$ and $a_i\in R$ for any $i\in\{0, \ldots,n-1\}$. We may assume that $n$ is the least integer such that $(*)$ is satisfied. Then, $n>1$ since $s\not\in R$. It follows that $(*)$ implies $s(s^{n-1}+\sum_{i=0}^{n-2}a _{i+1}s^{i})=-a_0\in R$. Since $s\neq 0$ because $s\in S\setminus R$ and $s^{n-1}+\sum_{i=0}^{n-2}a_{i+1}s^{i}\neq 0$ by the choice of $n$, it follows that either $s\in R$ or $s^{n-1}+\sum_{i=0}^{n-2}a_{i+1}s^{i}\in R$ because $R\subseteq S$ is semi-inert. In both cases, we get a contradiction, giving that $s$ is not integral over $R$. Then, $R\subseteq S$ is integrally closed.
  
Obviously, $\mathrm U(R)\subseteq\mathrm U(S)$. Assume, in addition, that $R\subseteq S$ is local. Then, $\mathrm U(S)\cap R=\mathrm U(R)$. Let $s\in\mathrm U (S)$. There exists $t\in\mathrm U(S)$ such that $st =1\in R\ (**)$, that is $s=t^{-1}$. Since $R\subseteq S$ is semi-inert, it follows that either $s$ or $t\in R$. Assume first that $s\in R$. Then, $s\in\mathrm U(S)\cap R=\mathrm U(R)$. Assume now that $s\not\in R$. We get that $t\in R\cap\mathrm U (S)=\mathrm U(R)$. Then, $t$ has an inverse in $R$, which is unique and is also its inverse in $S$, so that $s\in\mathrm U(R)$, a contradiction since $s\not\in R$. Hence, $\mathrm U(S)=\mathrm U(R)$ and $R\subseteq S$ is SL.
 \end{proof}
 
 We can  also build SL extensions with cyclotomic extensions.

 \begin{proposition} \label{7.4} Let $p$ be a prime integer different from 2 and such that 2 is a primitive root in $\mathbb Z/p\mathbb Z$. 
 
Set $K:=\mathbb Z/2\mathbb Z,\ L:=K[X]/(X^{p-1}+\cdots+1),\ R:=K[X]/(X^p-1)$ and $S:=K[X]/(X^{p+1}-X)$. The following properties hold:
  \begin{enumerate}
\item$X^{p-1}+\cdots+1$ is irreducible over $K$.
 \item There exists an injective ring morphism $f:R\to S$.
 \item $f$ is SL.
  \end{enumerate}
     \end{proposition}
\begin{proof} (1) comes from \cite[Theorem 2.47]{LN} because 2 is a primitive root in $\mathbb Z/p\mathbb Z$.

(2) Since $X^p-1=(X-1)(X^{p-1}+\cdots+1)$, the Chinese Remainder Theorem shows that $R\cong[K[X]/(X-1)]\times[K[X]/(X^{p-1}+\cdots+1)]\cong K\times L$ with $K\subseteq L$ a field extension of degree $p-1$. Similarly, $S\cong[K[X]/(X)]\times[K[X]/(X -1)]\times[K[X]/(X^{p-1}+\cdots+1)]\cong K^2\times L$. There is an injective ring morphism $g:K\times L\to K^2\times L$, given by $(x,y)\mapsto(x,x,y)$. Then, the following commutative diagram 
$$\begin{matrix}
 K\times L & \overset{g}\to & K^2\times L \\
 \uparrow &           {}           & \downarrow \\ 
      R       &  \overset{f}\to   &         S        
\end{matrix}$$
implies an injective ring morphism $f:R\to S$ gotten by 
$$f:R\to K\times L\to K^2\times L\to S$$ 
In fact, $g$ is defined as $g:=(g_1,g_2)$ where $g_1:K\to K^2$ is the diagonal map and $g_2$ is the identity on $L$. From now, we may identify $R$ with $K\times L,\ S$ with $K^2\times L$ and $f$ with $g$. 

Let $x$ be the class of $X$ in $L$, so that $L=K[x]$. Setting $y:=(0,x)\in R$, we get that $y^p=(0,1_L)$ and any element of $R$ is of the form $(a, b)$, where $a\in K=\{0,1_K\}$ and $b=\sum_{i=0}^{p-2}\alpha_ix^{i}\in L$ with $\alpha_i\in K$ for each $i$. Then, $(a,b)=(a,0)+(0,\sum_{i=0}^{p-2} \alpha_ix^{i})=(a,a)-(0,a)+(0,\sum_{i=0}^{p-2}\alpha_ix^{i})$ can be written as $a.1_R+\sum_{i=0}^{p-2}\beta_iy^{i}$, with $a$ and $\beta_i\in K$ for each $i$. It follows that $R=K[y]$. 

Setting $z:=(0,0,x)$, we get that $z=f(y)$, so that we have the injective ring morphism $ f:R\to S$ defined by $f(y)=z$. We may remark that $S\neq K[z]$ because $(0,1,0)\in S\setminus K[z]$. 

(3) Since $|\mathrm U(K)|=1$, we get that $|\mathrm U(R)|=|\mathrm U(K\times L)|=|\mathrm U(L)|=|\mathrm U(K^2\times L)|=|\mathrm U(S)|$. It follows that $f(\mathrm U (R))=\mathrm U(S)$ since $f(\mathrm U(R))\subseteq\mathrm U(S)$ and $f$ is  SL. 
\end{proof}
 
In the previous proposition the injective ring morphism $f:R\to S$ we built may not be unique as it is shown in the following example:

\begin{example} \label{7.5} Set $K:=\mathbb Z/2\mathbb Z,\ R:=K[X]/(X^3-1)$ and $S:=K [X]/(X^4-X)$. We will build two injective ring morphisms $f:R\to S$ which are  SL.
 
Let $y$ be the class of $X$ in $R$ and $t$ the class of $X$ in $S$, so that $R=K[y]$ and $S=K[t]$. We may use the proof of Proposition \ref{7.4} with $p=3$ since 2 is a primitive root in $\mathbb Z/3\mathbb Z$. Then, we get that $|\mathrm U(R)|=3$ giving $\mathrm U(R)=\{1,y,y^2\}$. Since there exists an injective ring morphism $f:R\to S$ which is SL, we also have $|\mathrm U(S)|=3$. As $f$ is also a linear morphism over the $K$-vector space $R$, we may define $f$ by the image of the basis $\{1,y,y^2\}$ of $R$ over $K$. An easy calculation of $(a+bt+ct^2+dt^3)^3=1$, 
 with $a,b,c,d\in K$ shows that $\mathrm U(S)=\{1,t^3+t+1,t^3+t^2+1\}$. Since we must have $f(1)=1$ and $f(y)\neq f(y^2)$ both in $\mathrm U(S)$, we get that $f(y)$ and $f(y^2)$ are the two different elements of $\{t^3+t+1,t^3+t^2+1\}$. Whatever the value we give to $y$, we get that $f (y^2)=f(y)^2$.

 So, such an $f$ is a ring morphism, which is obviously injective. At last, $f(\mathrm U (R))=\mathrm U(S)$ shows that $f:R\to S$ is SL. There are two such morphisms: $f_1$ and $f_2$ defined by $f_1(y)=t^3+t+1$ and $f_2(y)=t^3+t^2+1$. 
\end{example}
 
 \begin{proposition}\label{7.6} Let $R$ be a ring. Set $K:=\mathbb Z/2 \mathbb Z,\ T:=K^n\times R$ and $S=K^m\times R$, where $n<m$ are two positive integers. There is an injective ring morphism $f:T\to S$ given by $f(x,y):=(\varphi(x),y)$ where  $\varphi:K^n\to K^m$ is an injective ring morphism. Then, $f$ is SL.
       \end{proposition}
\begin{proof} Since $n<m$, we can build an injective ring morphism $\varphi:K^n\to K^m$ given, for instance, by $\varphi(x_1,\ldots,x_n)=(x_1,\ldots,x_n,x_n,\ldots,x_n)$ where the $m-n$ last terms are all equal to $x_n$. Then, the map $f:T\to S$ given by $f(x,y):= (\varphi(x),y)$ is an injective ring morphism. Since $K=\mathbb Z/2\mathbb Z$, it follows that $|\mathrm U(K)|=1=|\mathrm U(K^n)|=|\mathrm U(K^m)|$. But, $|\mathrm U(T)|=|\mathrm U(K^n)||\mathrm U(R)|$ and $|\mathrm U(S)|=|\mathrm U(K^m)||\mathrm U(R)|$, giving $|\mathrm U(T)|=|\mathrm U(S)|=|\mathrm U(R)|$, so that $f$  is SL.
\end{proof}

\end{document}